\documentclass[a4paper,reqno]{amsart}

\usepackage{amsfonts}
\usepackage{mathrsfs}
\usepackage{graphicx}
\usepackage[normalem]{ulem}
\usepackage{url}

\usepackage{multirow}
\usepackage{rotating}
\usepackage{booktabs}

\usepackage{amsthm,amssymb,amsmath,esint,amscd,latexsym,amsxtra}
\usepackage[colorlinks=true, urlcolor=blue,bookmarks=true,bookmarksopen=true,
citecolor=blue]{hyperref}
\usepackage{amscd}
\usepackage{enumerate,bbm}

\usepackage{fullpage}
\usepackage[all]{xy}
\usepackage[OT2,T1]{fontenc}
\usepackage{xspace}

\newcommand{\BA}{{\mathbb {A}}} 
\newcommand{\BC}{{\mathbb {C}}} 
 
\newcommand{\BG}{{\mathbb {G}}}

 \newcommand{\BN}{{\mathbb {N}}}
 
\newcommand{\BQ}{{\mathbb {Q}}}

 \newcommand{\BZ}{{\mathbb {Z}}}

\newcommand{\CE}{{\mathcal {E}}} \newcommand{\CF}{{\mathcal {F}}}
 \newcommand{\CH}{{\mathcal {H}}}
 
\newcommand{\CK}{{\mathcal {K}}} \newcommand{\CL}{{\mathcal {L}}}
 \newcommand{\CN}{{\mathcal {N}}}
\newcommand{\CO}{{\mathcal {O}}} 
 
\newcommand{\CS}{{\mathcal {S}}} \newcommand{\CT}{{\mathcal {T}}}
 
\newcommand{\CW}{{\mathcal {W}}}

\newcommand{\RU}{{\mathrm {U}}}

 \newcommand{\fn}{{\mathfrak{n}}}
 \newcommand{\fp}{{\mathfrak{p}}}
 
\newcommand{\fs}{{\mathfrak{s}}} \newcommand{\ft}{{\mathfrak{t}}}

\newcommand{\fM}{{\mathfrak{M}}}

 \newcommand{\fX}{{\mathfrak{X}}}

\newcommand{\Ad}{{\mathrm{Ad}}}

\newcommand{\Ann}{{\mathrm{Ann}}}
\newcommand{\Aut}{{\mathrm{Aut}}}

\newcommand{\diag}{{\mathrm{diag}}}
 
\newcommand{\End}{{\mathrm{End}}} 

\newcommand{\Frac}{{\mathrm{Frac}}}

 \newcommand{\GL}{{\mathrm{GL}}}

\newcommand{\Hom}{{\mathrm{Hom}}}

\newcommand{\id}{{\mathrm{id}}}

\newcommand{\Ker}{{\mathrm{Ker}}}

\newcommand{\ord}{{\mathrm{ord}}} 
\newcommand{\PGL}{{\mathrm{PGL}}}  \newcommand{\PL}{\mathrm{PL}}

\newcommand{\Rat}{{\mathrm{Rat}}}

\newcommand{\SL}{{\mathrm{SL}}}
\newcommand{\Spec}{{\mathrm{Spec\hspace{2pt}}}} 
\newcommand{\SO}{{\mathrm{SO}}}\newcommand{\Sp}{{\mathrm{Sp}}}

\newcommand{\val}{{\mathrm{val}}}
\newcommand{\Vol}{{\mathrm{Vol}}}

\newcommand{\matrixx}[4]{\begin{pmatrix}
		#1 & #2 \\ #3 & #4
\end{pmatrix} }        

\newcommand{\wt}[1]{{\widetilde {#1}}}
\newcommand{\wh}[1]{{\widehat {#1}}}

\newcommand{\ov}[1]{{\overline{#1}}}

\newcommand{\sk}{\medskip}
\newcommand{\lra}{\longrightarrow}

\newcommand{\ra}{\rightarrow} 
\newcommand{\bs}{\backslash}

\newcommand{\s}{\sk\noindent}

\newcommand{\dfn}[1]{\textit{#1}}

\makeatletter\def\varW@#1#2{%
\vtop{\m@th\ialign{##\cr
\hfil$#1 \mathrm{colim} $\hfil\cr
\noalign{\nointerlineskip\kern1.5\ex@}#2\cr
\noalign{\nointerlineskip\kern-\ex@}\cr}}
}
\makeatother
\makeatletter
\def\colim{%
\mathop{\mathpalette\varW@{}}\nmlimits@
}\makeatother

\theoremstyle{plain}
\newtheorem{thm}{Theorem}[section] \newtheorem{cor}[thm]{Corollary}

\newtheorem{lem}[thm]{Lemma}  \newtheorem{prop}[thm]{Proposition}
\newtheorem {conj}[thm]{Conjecture}

\theoremstyle{remark} \newtheorem{remark}[thm]{Remark}
\theoremstyle{definition} 
\theoremstyle{definition} \newtheorem{example}[thm]{Example} 
\newtheorem{defn}[thm]{Definition}\newtheorem{defn+lem}[thm]{Definition and Lemma}

\numberwithin{equation}{section}

\newcommand{\Map}{\mathrm{Map}}

\newcommand*{\sheafhom}{\mathrm{H}\kern -.5pt om}

\begin{document}
\title{Families of Canonical Local Periods on Spherical Varieties}

\author{Li Cai}
\address{Academy for Multidisciplinary Studies\\
Beijing National Center for Applied Mathematics\\
Capital Normal University\\
Beijing, 100048, People's Republic of China}
\email{caili@cnu.edu.cn}

\author[Corresponding author]{Yangyu Fan}
\address{Academy for Multidisciplinary Studies\\
Beijing National Center for Applied Mathematics\\
Capital Normal University\\
Beijing, 100048, People's Republic of China}
\email{b452@cnu.edu.cn}

\maketitle

\begin{abstract}
We consider the variation of canonical local periods on spherical varieties proposed by   Sakellaridis-Venkatesh in families. We formulate
 conjectures for the rationality and meromorphic property of canonical local periods and establish these conjectures for  strongly tempered spherical $G$-varieties without type $N$-roots when $G$ is split. 
\end{abstract}

\tableofcontents

\section{Introduction}

\subsection{Zeta integrals in families}The variation of local zeta integrals for  smooth representations
in families is extensively considered in literature and has various applications.

In the influential paper \cite{CPS10},  Cogdell-Piatetski-Shapiro consider the family 
$\{\pi|_x=I_P^{G}\sigma\otimes\chi|_x\mid x\in X\}$.
Here  $G$ is the product of two general linear groups over  a $p$-adic field $F$, $P\subset G$ is a parabolic subgroup with Levi factor $M$, $\sigma$ is a square-integrable irreducible (complex) $M(F)$-representation  with $I_P^G$ the normalized parabolic induction, $X$ is the (complex) torus parameterizing unramified characters of $M$ and $\chi|_x$ is the unramified character corresponding to $x$. Then they construct a Whittaker model $\CW(\pi)$  consisting of functions on $X\times G(F)$ which interpolates the Whittaker model of $\pi|_x$ for each $x\in X$ and show the Rankin-Selberg local zeta
integral $Z_{RS}(W,s)$ is meromorphic on $X\times\BC$ for each $W\in \CW(\pi)$. 
As an application,  this meromorphy property allows one to express $L(s,\pi)$ in term of $\sigma$.

Note that the family $\{\pi|_x\}$ is actually the specializations of some $\CO_X[G(F)]$-module $\pi$. Taking this algebraic viewpoint, one can  consider  smooth admissible $G(F)$-representations  over general reduced Noetherian coefficient  ring $R$, and establish  the theory  of zeta integrals, in particular the meromorphy property, in this algebraic setting.
For general linear groups, Moss \cite{Mos16, Mos16G}  algebraizes the Rankin-Selberg zeta integral machinery and obtains the  meromorphy property for  the co-Whittaker $R[G(F)]$-modules introduced by Emerton-Helm \cite{EH14} and Helm \cite{Hel16a}. For classical groups, Girsch  \cite{Gir21} algebraizes the doubling zeta integral machinery and establishes the meromorphy property.
One key   ingredient  is the admissibility of Jacquet modules of smooth admissible $R[G(F)]$-modules.

The variation of local zeta integrals in families plays an important role in the study of $p$-adic $L$-functions. For example, Disegni  \cite{Dis20,Dis21}  applied the meromorphy property of local zeta integrals at places $v \nmid p$  to construct the desired $p$-adic $L$-function appearing in the representation theoretic $p$-adic Gross-Zagier formula of ordinary families.
\subsection{Canonical local periods in families}
In a broad context, Sakellaridis-Venkatesh \cite{SV} introduces the canonical period on spherical varieties to study the special (local) $L$-value, which in some sense can be viewed as a vast generalization of local zeta integrals.

In this paper, we shall consider the variation  of canonical local periods in families, which we expect to behave similarly as zeta integrals.
  In this introduction, we shall only consider the geometric quotient  $Y:=H\bs G$ for reductive groups $H\subset G$ over $F$  which is   spherical and {\em strongly tempered} in the sense that  the matrix coefficients of any tempered $G(F)$-representation  are absolutely 
  integrable over $H(F)$.
  Then  the canonical local period $\alpha_\pi$ for any tempered
$G(F)$-representation $\pi$ associated to $Y$
is the bilinear form in $\Hom_{H(F)^2}(\pi \boxtimes \pi^\vee,\BC)$ defined by integrating matrix coefficients
of $\pi$ over $H$
\[\alpha_\pi(\varphi_1,\varphi_2) = \int_{H(F)} \langle \pi(h)\varphi_1,\varphi_2 \rangle dh,\quad
\varphi_1 \in \pi, \varphi_2 \in \pi^\vee.\]
Here $\pi^\vee$ is the contragredient representation of $\pi$ and $\langle \cdot,\cdot \rangle$ is a non-degenerate $G(F)$-invariant pairing on $\pi \times \pi^\vee$. For example (\cite[Proposition 4.10]{Zha14}),  in the case $Y = \Delta \GL_n \bs \GL_n \times \GL_{n+1}$,
\[\alpha_\pi\left( W_1,W_2 \right) = Z_{RS}\left( W_1,\frac{1}{2} \right) Z_{RS}\left( W_2,\frac{1}{2} \right), \quad
W_1 \in \CW(\pi,\psi), W_2 \in \CW(\pi^\vee,\psi^{-1})\]
for any tempered  $\GL_n(F) \times \GL_{n+1}(F)$-representation $\pi$  under proper normalizations on the Haar measure and the $G$-invariant pairing $\langle-,-\rangle$. Here, $\CW(\pi,\psi)$ is the Whittaker model of $\pi$
with respect to a nontrivial additive character $\psi$.

 To fix ideas, we consider the following set-up. Let $E$ be a field embeddable into $\BC$,
$R$ be a reduced $E$-algebra of finite type and $\Sigma\subset \Spec(R)$ be a fixed Zariski dense subset of closed points.
Let $\pi$ be a  finitely generated smooth admissible torsion-free left $R[G(F)]$-module (See Section  \ref{sec-modules} for details).
Assume that
\begin{enumerate}[(a)] 
\item for any $x\in\Sigma$, the specialization $\pi|_x:=\pi \otimes_{R} k(x)$ is absolutely irreducible and  tempered  in the sense that the following set of embeddings is nonempty $$\CE(\pi|_x):=\left\{\tau:\ k(x)\hookrightarrow\BC \big| \ (\pi|_x)_\tau:=\pi|_x\otimes_{k(x),\tau}\BC\ \mathrm{is}\ \mathrm{tempered}\right\};$$
	\item there exists a  finitely generated smooth admissible torsion-free $R[G(F)]$-module $\tilde{\pi}$ and a $G(F)$-invariant $R$-bilinear pairing
	\[\langle \cdot,\cdot \rangle: \pi \times \tilde{\pi} \lra R\] such that for each $x\in\Sigma$, 	$\tilde{\pi}|_x\cong (\pi|_x)^\vee$ and $\langle\cdot,\cdot\rangle$  induces a non-degenerate $G(F)$-invariant pairing $$\langle \cdot,\cdot \rangle|_x:\ \pi|_x\times\tilde{\pi}|_x\to k(x).$$
\end{enumerate}
This set-up is mainly motivated by our intended global application to study  $p$-adic $L$-functions on  eigenvarieties: $\Spec(R)$ (resp. $\Sigma$) plays the role of eigenvariety (resp. classical points), $\pi$ is the local component of the universal automorphic representation on the eigenvariety and  the above assumptions hold in many interesting cases (See \cite[Chapter 4]{Dis19} for a  detailed discussion for the case $G=\GL_2$). 
\begin{conj}[Meromorphy and rationality for canonical period integrals]\label{conj-2} Assume $Y$ is a strongly tempered spherical $G$-variety.
	Let $\pi$
	be a $R[G(F)]$-module as above. Then
	\begin{enumerate}
		\item (Rationality) For any $x\in\Sigma$, there exists a unique bi-$H(F)$-invariant pairing
		$$\alpha_{\pi|_x}:\ \pi|_x\times \tilde{\pi}|_x\to k(x)$$
	such that for any $\tau\in\CE(\pi|_x)$,	 the following diagram commutes
			\[\xymatrix{
			&\pi|_x \times \tilde{\pi}|_x \ar[d]^-{\tau} \ar[r]^-{\alpha_{\pi|_x}} & k(x) \ar[d]^-{\tau} \\
			&(\pi|_x)_\tau \times (\tilde{\pi}|_x)_\tau \ar[r]^-{\alpha_{(\pi|_x)_\tau}} & \BC
			.}\]
		 Here $\alpha_{(\pi|_x)_\tau}$ is defined with respect to the linear extension of $\langle\cdot,\cdot\rangle|_x$.
		
		\item (Meromorphy) Up to shrinking $\Spec(R)$ to an open subset  containing $\Sigma$, there exists a bi-$H(F)$-invariant $R$-linear pairing
		$$\alpha_\pi:\ \pi\times\tilde{\pi}\to R$$
		such that for any $x\in\Sigma$, the following
		diagram is commutative
		\[\xymatrix{
			&\pi \times \tilde{\pi} \ar[d] \ar[r]^-{\alpha_\pi} & R \ar[d] \\
			&\pi|_x \times \tilde{\pi}|_x \ar[r]^-{\alpha_{\pi|_x}} & k(x)
			.}\]
		Here both vertical arrows are the specialization maps. 
 \end{enumerate}
 \end{conj}
For general spherical varieties $Y$, one can extract the canonical local period $\alpha_\pi$  from the
conjectural Plancherel decomposition for the unitary  $G(F)$-representation $L^2(Y)$. Moreover, to make the canonical local period independent of the invariant pairing $\langle \cdot,\cdot \rangle$, we consider the quotient of
$\alpha_\pi$ by $\langle \cdot,\cdot \rangle$. In this way, the existence
of the $G(F)$-invariant $R$-bilinear pairing in condition
(b) of Conjecture \ref{conj-2} can be dropped.
For the precise formulation, we refer to Conjectures \ref{Rationality Q} and \ref{local-conj}  for details.

The following is the main result of this paper.

\begin{thm}[Theorem \ref{contro} and Theorem \ref{crucial}]\label{main}
Let $Y=H\bs G$ be a homogeneous spherical $G$-variety with $G$ split.
Assume $Y$ is  
strongly tempered  without type $N$-roots. Let $\pi$ be a $R[G(F)]$-module as in Conjecture \ref{conj-2}. Then
\begin{itemize}
    \item the rationality conjecture holds for $\pi$,
    \item the meromorphy conjecture holds for $\pi|_U$ where $U\subset\Spec(R)$ is a Zariski dense open subset;
      moreover the meromorphy conjecture holds for $\pi$  if the discrete support of  $\pi$ is {\em rigid} around each $x\in\Sigma$ (see Definition \ref{locally twisting}).
\end{itemize}
\end{thm}
Roughly speaking, the discrete support of $\pi$ is rigid around $x\in\Sigma$ means that  if $\pi\hookrightarrow I_P^G\sigma$ for some discrete series $\sigma$, then for $y\in\Sigma$ around $x$,  $\pi|_y\hookrightarrow I_P^G(\sigma\otimes\chi|_y)$ for some  character $\chi|_y$.  We expect this property to hold in a reasonable generality. In fact,  if  $G$ is a product of general linear groups, the discrete support of $\pi$ is rigid around each $x\in\Sigma$  for all the $R[G(F)]$-modules  $\pi$
 in consideration (see  Proposition \ref{GL(n)}) and for general  $G$, the discrete support of $\pi$ is rigid around $x\in\Sigma$ if $\pi|_x=I_P^G\sigma$ with $\sigma$  regular supercuspidal  (see Proposition \ref{quasi-split}).

We record an immediate corollary, which studies \dfn{distinguished} representations  in  an unramified twisting family. 
Let $P=MN\subset G$ be a parabolic subgroup with Levi factor M and  denote by $X$
the torus of unramified complex characters  on $M(F)$.

\begin{cor}Let $Y=H\bs G$ be a homogeneous spherical $G$-variety with $G$ split. 
Assume $Y$ is strongly tempered without type $N$-roots and wavefront. 
Let $\sigma$ be a  discrete series (complex) representation on $M(F)$ and assume the parabolic induction 
 $ I_P^G \sigma$ is irreducible and  
 distinguished, i.e. $m(I_P^G\sigma) := \dim \Hom_{H(F)}(I_P^G\sigma,\BC)\neq0$.
Then all the representations $\pi|_x = I_P^G \sigma \otimes \chi|_x$ (not necessarily irreducible), $x \in X$ in the unramified twisting family are distinguished. Here for $x\in X$, $\chi|_x$ is the unramified character corresponding to $x$.
 \end{cor}
\begin{proof}Fix a nonzero invariant pairing $\langle \cdot,\cdot \rangle_\sigma$ on	$\sigma \boxtimes \sigma^\vee$ and choose  a good maximal open compact subgroup $K\subset G(F)$ such that $G(F) = P(F)K$. Then by  the identifications 
\[\pi|_x  \stackrel{\sim}{\lra} V := I_{K \cap P}^K \sigma\ (\mathrm{resp.}\ (\pi|_x)^\vee  \stackrel{\sim}{\lra} V^\vee := I_{K \cap P}^K \sigma^\vee), \quad f \mapsto f|_K\]
and \cite[Proposition 3.1.3]{Cas95}, 
the bi-linear pairing
\[V\times V^\vee\to\BC;\quad  \langle v_1,v_2 \rangle = \int_K \langle v_1(k), v_2(k) \rangle_\sigma dk\]
induces a non-degenerate invariant pairing on $\pi|_x\times (\pi|_x)^\vee$ for each $x\in X$. 
Moreover, for any $v_1 \in V$, $v_2  \in V^\vee$ and any $g \in G$, the function 
   \[x \mapsto \langle \pi|_x(g)v_1,v_2 \rangle\]
 is regular. In particular,  the family $\pi|_x$ satisfies the assumptions in Conjecture \ref{conj-2}. 
 
 By \cite[Theorem 6.4.1]{SV},  $\alpha_{I_P^G\sigma}\neq0$ since $I_P^G\sigma$ is distinguished. By Theorem \ref{main}, the canonical local period is meromorphic and thus  $\alpha_{\pi|_x)} \not= 0$
	for $x$ in an open subset of $X$. In this open subset,  $m(\pi|_x)\geq1$ and consequently by  the upper-semicontinuity of $m(\pi|_x)$ established in \cite[Appendix D]{FLO}, one deduces $m(\pi|_x) \geq 1$ for all $x\in X$.
\end{proof}

\begin{remark}
   In \cite{BD08}, for symmetric pairs $(G,H)$, meromorphic families of $H$-distinguished linear functionals
   over complex torus are constructed. By generalizing this result to reduced Noetherian coefficient rings and
   comparing the square of these linear functionals with the canonical local periods, one may obtain the meromorphy
   of canonical period integrals (Conjectures \ref{Rationality Q} and \ref{local-conj}) for symmetric pairs.
   We plan to consider this problem in another paper.
\end{remark}

\begin{remark}
Attached to the canonical local period $\alpha_\pi$, there is the spherical character
$$J_\pi(f):=\sum_i \alpha_\pi(\pi(f)\varphi_i\otimes\varphi^i),\quad f\in\CS(G(F),\BC)$$
on the Hecke algebra $\CS(G(F),\BC)$ of complex-valued compactly supported functions on $G(F)$.
Here $\{\varphi_i\}$ and $\{\varphi^i\}$ are arbitrary dual basis of $\pi$ and $\pi^\vee$
with respect to $\langle\cdot,\cdot\rangle$. Sometimes (e.g. in the relative trace formula framework),
it is more convenient to consider the spherical characters.

In an old version of this paper, we also formulate parallel conjectures
for spherical characters (See \cite[Conjecture 3.12 and Conjecture 3.13]{CF21}).
However, we don't know whether the two versions (local periods/spherical characters) are equivalent.
Under the condition that {\em the fibre ranks of representations are locally constant},
the conjecture for canonical local periods implies that for spherical characters (See \cite[Lemma 3.14]{CF21}).

In that version, we obtain the meromorphy property for spherical characters
in the Gan-Gross-Prasad unitary group case
(See \cite[Theorem 5.8]{CF21}) using a completely different method.
We apply the spherical character identity of Beuzart-Plessis to relate the
spherical character  on unitary groups to the
Rankin-Selberg and Asai zeta integrals. The meromorphy property follows from that of Rankin-Selberg and
Asai zeta integrals.
\end{remark}

Before explaining the strategy to Theorem \ref{main}, we  explain our global motivation.

Let  $F$ be a number field and  $H\subset G$ be a spherical pair of reductive groups over $F$
satisfying certain conditions (e.g. the multiplicity one property).
For any  cuspidal unitary automorphic representation $\pi$
 of $G(\BA)$,   Sakellaridis-Venkatesh \cite{SV} conjectured that the global period
 \[P_H(\varphi) = \int_{H(F) \bs H(\BA)} \varphi(h) dh, \quad \varphi \in \pi\]
  has a  conjectural Eulerian decomposition into products of local periods
\label{Global conj}
\[\tag{$*$} P_H(\varphi_1)P_H(\varphi_2)=
L_Y^{(S)}(\pi) \prod_{v \in S} \alpha_{\pi_v}\left( \varphi_{1,v}\otimes\varphi_{2,v}
\right), \quad \varphi_1 \in \pi,
\varphi_2 \in \pi^\vee\]
where $L_Y^{(S)}(\pi)$ is certain $L$-value associated to $Y:=H\backslash G$ and $\pi$. Examples of
such decomposition include
the Ichino formula for the triple product case
 (see \cite{Ich08}) and the Ichino-Ikeda conjecture  for the unitary Gan-Gross-Prasad case (proved in \cite{Iso,BCZ} recently).

 The conjectural decomposition $(*)$ can be used to  construct $p$-adic $L$-functions. Using geometric methods, one can study the variation of (proper modification of) the global periods $P_H(\varphi)$ 
 when $\pi$ varies in a $p$-adic family. Instead of choosing the test vector $\varphi$  carefully (so that the local period can be computed explicitly) and then viewing the period integral $P_H(\varphi)$ as the $p$-adic $L$-function (see \cite{BDP13} for the Waldspurger toric case, \cite{Hsi17} for the triple case, 
 \cite{Harris} for the Gan-Gross-Prasad case), we prefer to apply the meromorphy property of canonical local periods at places $v\nmid p$ to view  the ratio
 \[ \frac{P_H(\varphi_1)P_H(\varphi_2)}{
 	  \prod_{v \in S, v \nmid p}
 	\alpha_{\pi_v}\left( \varphi_{1,v} \otimes \varphi_{2,v} \right)},
 \quad \varphi_1 \in \pi, \varphi_2 \in \pi^\vee,\]
which is independent of the choice of test vectors, as the candidate for the conjectural  $p$-adic $L$-function $\CL_{Y}^{(S)}(\pi)$ of the $p$-adic family $\pi$.  In the Waldspurger toric period case, Liu-Zhang-Zhang \cite{LZZ} implemented this strategy to
construct a $p$-adic $L$-function and established a special value formula of this $p$-adic $L$-function.

In some cases (including the triple product and the Gan-Gross-Prasad case),
the conjectural decomposition $(*)$ has an arithmetic counterpart, which
relates the Beilinson-Bloch height pairing of certain cycles to the product of certain derivative $L$-values and the same local periods. In a similar vein, this arithmetic formula can be used to study  the  derivative special values of $p$-adic $L$-functions of different nature. For results in this direction, see  \cite{Dis19}  for the $p$-adic Gross-Zagier formula.


\subsection{Sketch of the proof} Now we return to the proof of Theorem \ref{main}.
Let $\pi$ be a finitely generated  smooth admissible  torsion-free $R[G(F)]$-module as in  Conjecture \ref{conj-2}. In the same spirit as \cite{Mos16, Mos16G, Gir21},
the basic strategy to prove Theorem \ref{main} is reducing the
meromorphy of $\alpha_\pi$ to showing certain formal  power series are actually rational functions. For this, we study the asymptotic
behavior of the matrix coefficients of $\pi$ along the spherical subgroup $H$. Two key ingredients in the proof are
\begin{itemize}
	\item  Admissibility of Jacquet modules (\cite[Corollary 1.5]{DHK22}, see also Theorem \ref{admjac}).
		For any parabolic subgroup $P=MN\subset G$ with Levi factor $M$, the Jacquet module
		$$J_N(\pi):=\pi/ \pi(N),\quad \pi(N):=\langle\pi(n)v-v|v\in\pi, n\in N(F)\rangle\subset \pi$$
                is a smooth admissible $R[M(F)]$-module.
	\item  Asymptotic behavior of matrix coefficients (Corollary \ref{Asym}). This is a consequence
		of Casselman's canonical pairing theorem for $G(F)$-representations (recorded as Theorem \ref{cano}).  Fix a minimal parabolic subgroup
		$P_{\emptyset_G}=M_{\emptyset_G} N_{\emptyset_G}\subset G$ with $A_{\emptyset_G}\subset M_{\emptyset_G}$
		a maximal split torus. Let $\Delta_G$ be the set of simple $G$-roots with respect to $(P_{\emptyset_G},A_{\emptyset_G})$. 
        Then for any $\Theta_G\subset \Delta_G$,  any $v\in\pi(N_{\Theta_G})$ and any $\wt{v} \in \wt{\pi}$,
		there exists $\epsilon>0$ such that  for any $a\in A_{\emptyset_G}^-$ with
		$|\alpha(a)|<\epsilon$ for all $\alpha\in \Delta_G-\Theta_G$, $\langle \pi(a)v,\wt{v} \rangle=0$ where 	$$A_{\emptyset_G}^-:=\{a\in A_{\emptyset_G}(F)|\ |\alpha(a)|\leq 1\ \forall\ \alpha\in \Delta_G\}.$$
\end{itemize}
We  now take the triple product case and the  $(\GL_3 \times \GL_2,\GL_2)$-case to explain the ideas.

We first consider the (split) triple product case: $H:=\BG_m\backslash \GL_2 \subset
G:=\BG_m\backslash \GL_2^3$. Here, $H$ embeds into $G$ diagonally. 
Note that for any $x \in \Sigma$, by the Cartan decomposition  (Theorem \ref{Cartan})
$$\GL_2(F)=\bigsqcup_{a\geq b}Kt(a,b) K, \quad K = \GL_2(\CO_F),\quad t(a,b) = \matrixx{\varpi^a}{0}{0}{\varpi^b}  $$
the canonical local period
\[\alpha_{\pi|_x}(v,\wt{v}) = \sum_{a \geq 0} \Vol(Kt(a,0)K)
\iint_{K \times K} \langle \pi|_x(k_1t(a,0) k_2) v,\wt{v} \rangle|_x dk_1dk_2.\]
with  $$\Vol(Kt(a,0)K)=\begin{cases} (1+q^{-1}) q^{a} &\ a>0\\ 1 &\ a=0.\end{cases}$$   
Here, we assume $\Vol(K) = 1$ and $q$ is the cardinality of the residue field of $\CO_F$. 

It suffices to show that  for any $v\in \pi$ and $\wt{v} \in \wt{\pi}$, the formal sum 
$$\wt{F}_{\emptyset_H,v,\wt{v}}(T):=\sum_{a\geq 1}\langle \pi(t(a,0))v, \wt{v} \rangle T^a$$
is actually  rational over $R$ and then evaluate at $T=q$.

Identify the set of simple roots $\Delta_G$ with $\Delta_{\GL_2}\sqcup \Delta_{\GL_2} \sqcup \Delta_{\GL_2}$, where  $\Delta_{\GL_2}$  consists of
$$\alpha:\ A_2(F)=(F^\times)^2\to F^\times,\quad (t_1,t_2)\mapsto t_1/t_2.$$
Here $A_n\subset \GL_n$ is the diagonal torus. Consider $\Theta_G=\emptyset_G$ in the above theorem for asymptotic behavior of matrix coefficients.
Since $ A_{\emptyset_G} =\Delta\BG_m\bs A_2\times A_2 \times A_2$, one finds that $\wt{F}_{\emptyset_H,v,\wt{v}}(T)$ is a polynomial for  $v\in \pi(N_{\emptyset_G})$.
Moreover,
 $$\pi(t(1,0)) \in \End_{R[A_{\emptyset_G}]}J_{N_{\emptyset_G}}(\pi) \hookrightarrow \End_R J_{N_{\emptyset_G}}(\pi)^{K_{A_{\emptyset_G}}}$$
where $K_{A_{\emptyset_G}}$ is any open compact subgroup of $A_{\emptyset_G}(F)$ such that $J_{N_{\emptyset_G}}(\pi)^{K_{A_{\emptyset_G}}}$ containing a (finite) system of generators of $J_{N_{\emptyset_G}}(\pi)$.
By the admissibility of Jacquet modules, $J_{N_{\emptyset_G}}(\pi)^{K_{A_{\emptyset_G}}}$ is finitely generated over $R$.
Thus  there exists a non-zero polynomial $P(X)=\sum_{n\geq0}c_nX^n\in R[X]$
such that for any $v\in \pi$, $P(\pi(t(1,0)))v\in \pi(N_{\emptyset_G})$. Since
$$ \wt{F}_{\emptyset_H,P(\pi(t(1,0)))v,\wt{v}}(T)=P(T^{-1})\wt{F}_{\emptyset_H,v,\wt{v}}(T)-\sum_{n\geq1}\sum_{a=1}^nc_n\langle
\pi(t(a,0))v,\wt{v}\rangle T^{a-n}\in R[T],$$
one immediately deduces $F_{\emptyset_H,v,\wt{v}}(T)$ is a rational function of the form $\frac{Q(T)}{P(T^{-1})}$ with $Q(T)\in R[T]$.

 At each point $x\in\Sigma$, the evaluation $\wt{F}_{\emptyset_H,v,\wt{v}}|_x(q)$ of $\wt{F}_{\emptyset_H,v,\wt{v}}|_x$ at $T=q$ gives the desired value, namely  the absolutely convergent sum $$\sum_{a>0}\langle \pi(t(a,0))v,\wt{v}\rangle q^a,$$
by the absolutely convergence of $\alpha_{\pi|_x}(v,\wt{v})$. When the discrete support of $\pi$ is rigid at $x\in\Sigma$, we explicitly construct the polynomial $P(T)$ and then deduce $P|_x(q^{-1})\neq0$ from the absolutely convergence of $\alpha_{\pi|_x}(v,\wt{v})$. Consequently, 
  $\wt{F}_{\emptyset_H,v,\wt{v}}(q)|_x=\wt{F}_{\emptyset_H,v,\wt{v}}|_x(q)$ and the meromorphy conjecture for $\pi$ holds.

Now we turn to the  rank two case
$$H=\GL_2\hookrightarrow G:=\GL_3\times \GL_2,\quad g\mapsto \left(\begin{pmatrix} g & 0\\ 0 & 1\end{pmatrix}, g\right).$$
Again by the  Cartan decomposition and the volume formula, it suffices to show that for any $v\in \pi$ and $\wt{v}\in \wt{\pi}$,
the formal power series
\[\wt{F}_{\emptyset_H,v,\wt{v}}^{\pm}(T_1,T_2):=\sum_{a > b, \pm b>0} \langle \pi(t(a,b))v,\wt{v} \rangle T_1^{a-b} T_2^{|b|},\quad\quad \wt{F}_{\Delta_H,v,\wt{v}}^{\pm}(T_2):=\sum_{\pm a>0} \langle \pi(t(a,a))v,\wt{v} \rangle T_2^{|a|} \]
\[\wt{F}_{\emptyset_H,v,\wt{v}}^{0}(T_1):=\sum_{a > 0} \langle \pi(t(a,0))v,\wt{v} \rangle T_1^{a} ,\quad\quad \wt{F}_{\Delta_H,v,\wt{v}}^{0}:=\langle v, \wt{v}\rangle \]
are all rational over $R$ and  then sum up the evaluations at $T_1=q$ and $T_2=1$ . Here $\tilde{F}^\pm_{\emptyset_H,v,\wt{v}}$ deals with summations over the  ``2-dimensional cones'' $\boxed{a>b, \pm b>0}$, while 
$\tilde{F}^0_{\emptyset_H,v,\wt{v}}$ and $\wt{F}_{\Delta_H,v,\wt{v}}^{\pm}$ deal with the ``1-dimensional cones'' 
$\boxed{a>b, b=0}$ and  $\boxed{a=b,\pm b>0}$.

The evaluation part is similar to the triple case. For the rationality part, we proceed by induction.  We reduce the rationality of the series over ``2-dimensional cones'' to
that of the series over ``1-dimensional cones'' which can be handled similarly as the triple product case.

Identify $\Delta_G$ with
$\Delta_{\GL_3}\sqcup \Delta_{\GL_2}$ where $\Delta_{\GL_3}=\{\beta_1,\beta_2\}$ with
$$\beta_i:\ A_3(F)=(F^\times)^3\to F^\times;\quad (t_1,t_2,t_3)\mapsto t_i/t_{i+1}.$$
Consider the subset $\Theta_G = \{\beta_1,\alpha\}$ of $\Delta_G$ and
$\pi(t(1,1))\in \End_{R[M_{\Theta_G}(F)]}J_{N_{\Theta_G}}(\pi)$. Then by the aforementioned asymptotic behavior of matrix coefficients,  the rationality of
$\wt{F}^+_{\emptyset_H,v,\wt{v}}$ reduces to that of $\wt{F}^0_{\emptyset_H,v,\wt{v}}$. And the rationality of $\wt{F}^0_{\emptyset_H,v,\wt{v}}$ can be
obtained by considering the subset $\{\beta_2,\alpha\}$ and the endomorphism induced by $\pi(t(1,0))$.

To deal with $\wt{F}^-_{\emptyset_H,v,\wt{v}}$, we furthur decompose it  into formal series over the ``2-dimensional cones'' 
$ \boxed{a\geq 0>b}$ and $\boxed{0>a>b}$.
Then by considering the action of the Weyl group of $G$,
one can reduce the rationality of the series over $\boxed{a\geq 0>b}$ and $\boxed{0>a>b}$ to that of $\wt{F}^+_{\emptyset_H,v,\wt{v}}$.

The strategy  is available for  very general situations. We introduce the  notion of \dfn{reduction structures} in Definition \ref{controllable} which abstracts the geometric properties validating
the reduction process above. Then we show the meromorphic property of $\alpha_\pi$ holds for all  strongly tempered spherical $G$-varieties admitting  reduction structures in Theorem \ref{contro}. Finally in Theorem \ref{crucial}, we prove  all split strongly tempered spherical varieties without type $N$-roots admit reduction structures  case by case using the classification result (listed in \cite[Section 1.1]{WZ21}).

\begin{remark}Actually, inner forms and low rank variations of split strongly tempered spherical $G$-varieties without type $N$-roots   also admit reduction structures. See Proposition
\ref{ggp-unitary} for the  Gan-Gross-Prasad case for unitary groups.
\end{remark}
\s{\bf Acknowledgement} We express our sincere gratitude to Prof. Y. Tian for his consistent encouragement, and to Prof. A. Burungale for carefully reading earlier versions of the article. We are grateful to the anonymous referee whose comments helped us improve the article. In particular, we thank the referee for pointing out a  gap in an earlier version and Prof. G. Moss for generously suggesting Proposition \ref{quasi-split} to fix it.  We thank Prof.  J.-F. Dat for kindly answering our question on  Jacquet modules and thank Prof. J. Yang for discussions on  global applications.

L. Cai is partially supported by NSFC grant No.11971254.

\section{Notation and conventions}
In the rest of this paper,
\begin{itemize}
	\item  $F$ is a finite extension of $\BQ_p$ with ring of integers $\CO_F$,  uniformizer $\varpi$,  residue field $k$ and normalized norm  $$|\cdot|:\ F\to (\sharp k)^{\BZ}\cup\{0\};\quad \varpi\mapsto q^{-1},\quad q=\sharp k.$$
	\item  $E$ is a  field  embeddable into $\BC$ and $\bar{E}$ is an algebraic closure of $E$,
	\item  $X/E$ is a locally of finite type reduced  scheme and for any $x\in X$,  $k(x)$ is the residue field $\CO_{X,x}/m_x\CO_{X,x}$,
	\item  $\Sigma\subset X$ is a Zariski dense subset of closed points.
\end{itemize}
\subsection{Fibres of quasi-coherent sheaves}\label{fibers}
Let $\CF$ be a quasi-coherent sheaf on $X$. A non-zero section  $s\in \CF(X)$ is {\em torsion} if there exists an open affine subset $U=\Spec(R)\subset X$ such that $\Ann_R(s|_U)$ contains non-zero divisors. The sheaf  $\CF$ is {\em torsion-free} if $\CF(U)$ contains no torsion elements  for any open affine subset  $U\subset X$.

For any $x\in X$, the fiber $\CF{|_x}$ of $\CF$ at  $x$ is the $k(x)$-vector space $\CF_x \otimes_{\CO_{X,x}} k(x)$. For any $s\in \CF(X)$, let $s(x)\in \CF|_x$ be the image of $s$ under the natural map  $\CF(X)\to \CF_x\to\CF{|_{x}}$.  When  $\CF$ is coherent,
\begin{itemize}
	\item the function $$\phi_{\CF}:\ X\to\BN,\quad x\mapsto \dim_{k(x)}\CF{\mid_x}$$ is upper semi-continuous, i.e. $\{x\in X\mid \phi_{\CF}(x)\leq n\}$ is open for any $n\in \BN$ (The upper semi-continuous theorem, see \cite[Example 12.7.2]{Har77}).
	\item if $\phi_{\CF}$ is locally constant,  $\CF$ is finite projective (note that $X$ is reduced, see \cite[\href{https://stacks.math.columbia.edu/tag/0FWG}{Lemma 0FWG}]{SP}),
	\item the {\em zero locus} $\{x\in X|s(x)=0\}\subset X$ of a non-zero section $s\in\CF(X)$ is constructible by Chevalley theorem and hence $s$ is torsion if the zero locus contains a Zariski dense subset.
\end{itemize}

Let  $\CF^\star\subset \CF$ be an subsheaf. For any section $s\in\CF^\star(X)$, let $s_\CF(x)\in \CF|_x$  be the image of $s$ under the natural map  $\CF^\star(X)\hookrightarrow\CF(X)\to \CF_x\to\CF{|_{x}}$ and let $\CF^\star|_{\CF,x}\subset\CF|_x$ be the subset consisting of all $s_\CF(x)$. Note that when $\CF^\star\subset \CF$ is a $\CO_X$-submodule, then $\CF^\star|_{\CF,x}$ (resp. $s_\CF(x)$) is the image of  $\CF^\star|_x$ (resp. $s(x)$) via the natural map
$\CF^\star|_x\to\CF|_x$. When no confusion arises, we will omit $\CF$ in the notations.

 We introduce the following notations. Let  $\CK_X$ be the sheaf of total quotient rings of $\CO_X$. Note that  (see \cite[\href{https://stacks.math.columbia.edu/tag/02OW}{Lemma 02OW}]{SP}) for any $x\in X$, $\CK_{X,x}$ is the total quotient ring of $\CO_{X,x}$.
\begin{itemize}
	\item A family $\{f_x \in \CF^\star|_x\}_{x\in \Sigma}$ is called  {\em meromorphic} if there exists an open subset $X^\prime\subset X$ containing $\Sigma$ and a section $F\in \CF^\star(X^\prime)$ such that
	for any $x\in \Sigma$, $F(x)=f_x$.
	When $\CF$ is  torsion-free, the section $F$ is unique in the sense  that for any open subset $X^{\prime\prime}$ containing $\Sigma$ and any section $G\in \CF^\star(X^{\prime\prime})$ such that $G(x)=f_x$,
	$F\mid_{X^\prime\cap X^{\prime\prime}}=G\mid_{X^\prime\cap X^{\prime\prime}}.$
	\item    A family $\{f_x \in \Map(\CF^\star|_x,k(x)) \}_{x\in\Sigma}$
	is called {\em meromorphic} if there exists an open subset $X^\prime\subset X$  containing $\Sigma$ and a sheaf morphism $F:\ \CF^\star|_{X^\prime}\to\CK_{X^\prime}$ such that for any $x \in \Sigma$, $F_x(\CF^\star_x)\subset \CO_{X,x}$ and the following diagram commutes
	\[\xymatrix{
		\CF^\star_{x} \ar[d] \ar[r]^{F_x} & \CO_{X,x} \ar[d] \\
		\CF^\star|_x  \ar[r]^{f_x} & k(x)   }\]
	Such a morphism $F$ is unique and extends to $X$ by
	$$\CF^\star(X)\to\CF^\star(X^\prime)\xrightarrow{F} \CK_{X^\prime}(X^\prime)\cong \CK_X(X).$$
\end{itemize}
In both cases, we say $F$ {\em interpolates} $f_x$.

 For any $\CO_X^\times$-subsheaf $\CF^\star\subset \CF$, let $\CF^{\star,-1}$ denote the sheaf $\CF^\star$
with inverted $\CO_X^\times$-action
\[c\cdot f:= c^{-1}f, \quad c \in \CO_X^\times, \
f\in \CF^\star.\] Let  $Q_{\CF^\star}(\CF)$  denote the sheaf $\CF\times\CF^{\star,-1}$ and let $\frac{f}{g}\in Q_{\CF^\star}(\CF)(X)$   denote the section  $(f,g)\in Q_{\CF^\star}(\CF)(X)$.   We will omit  $\CF^\star$ in notations when no confusion arises.
\subsection{Measures on  $G$-varieties}
Let $G$ be a reductive group over $F$. A  $G$-variety $Z$ over $F$ is a geometrically integral and separated $F$-scheme of finite type together with an $F$-algebraic $G$-action $G\times Z\to Z$.

Assume  $Z$ is a smooth $n$-dimensional $G$-variety over $F$ equipped with a $G$-equivariant $F$-rational  differential $n$-form $\omega$. Cover $Z(F)$
by local charts $(U,\quad U\xrightarrow[\cong]{i_U} \CO_F^n)$. Note that the standard   differential $n$-form $dx_1\wedge\cdots\wedge dx_n$ on $F^n$ defines a Haar meausre $dx_1\cdots dx_n$ on $F^n$ such that $\Vol(\CO_F^n)=1$. Assume $$i_{U,*}(\omega|_U)=f(x_1,\cdots,x_n)dx_1\wedge\cdots \wedge dx_n.$$ Then  the measure $i_U^*(|f(x_1,\cdots,x_n)|_Fdx_1\cdots dx_n)$  on $U$  glues to a $G(F)$-equivariant measure $|\omega|$ on $Z(F)$.  Note that for any $\BQ$-algebra $A$,  $\int_{Z(F)}f|\omega|\in A$ for any $f\in\CS(Z(F),A)$,  the space of $A$-valued compact supported locally constant functions on $Z(F)$.

Let $H\subset G$ be a closed $F$-subgroup. By \cite[Section 3.8]{Sak08}, there exists a $G$-equivariant top degree differential form on the geometric quotient $H\backslash G$ if and only if the modulus characters of $G$ and of $H$ coincide  on $H$. In particular, when $H$ is reductive,  the above construction is applicable to $H\backslash G(F)$ and we will denote the resulting measure by $d\mu_{H\backslash G}$  when the top differential is insignificant.

\subsection{Smooth admissible modules}\label{sec-modules}
Let $G$ be a reductive group over $F$. For any Noetherian ring $R$, a  $R$-module $M$ equipped with a $R$-linear action of $G(F)$ is called
\begin{itemize}
	\item {\em smooth} if  any $v \in M$ is fixed by some open compact subgroup of $G(F)$;
	\item {\em admissible} if for any compact open subgroup $K \subset G(F)$, the submodule
	$M^K\subset M$ of $K$-fixed elements is  finitely generated  over $R$;
	\item {\em finitely generated} if $M$ is finitely generated as a $R[G(F)]$-module;
\item {\em torsion-free} if $M$ is torsion-free as a $R$-module.
\end{itemize}

A quasi-coherent $\CO_X$-module $\pi$ equipped with
a group morphism $G(F) \ra \Aut_{\CO_X}(\pi)$ is called an {\em $\CO_X[G(F)]$-module}.  It is moreover called  {\em smooth/admissible/finitely generated/torsion-free} if for any $x\in X$, there exists an open affine neighborhood $x\in U=\Spec(R)\subset X$ such that $\pi(U)$ is smooth/admissible/finitely generated/torsion-free as $R[G(F)]$-module. Note that the fiber $\pi|_x$ is a smooth/admissible/finitely generated $k(x)[G(F)]$-module for any $x\in X$ if so is $\pi$.

 When $X=\Spec(C)$ for some field extension $C/E$, a smooth $C[G(F)]$-module $\pi$ is called {\em irreducible} (resp. {\em absolutely irreducible}) if it  (resp. its base change to any (hence all) algebraically closed field) contains no proper non-zero submodule. Unless otherwise specified, all irreducible representation in this paper would be {\em non-zero}.

\begin{lem}\label{irre} Assume $X=\Spec(R)$ and let $\fp\subset R$ be a minimal prime with corresponding  generic point $\eta$. Let $\pi$ be a   finitely generated   smooth admissible torsion-free $R[G(F)]$-module.  Then if $\pi|_x$ is irreducible (resp. absolutely irreducible) for all $x\in\Sigma$,   $\pi|_\eta$ is irreducible (resp. absolutely irreducible). 
\end{lem}
\begin{proof} 
Let $\{\fp_1=\fp,\fp_2,\cdots,\fp_n\}$ be the set of minimal primes of $R$ and let $\sigma_i$ be the maximal torsion-free quotient of the $R/\fp_i[G(F)]$-module $\pi/\fp_i\pi$. By the argument of \cite[Lemma 6.3.6]{EH14}, the diagonal map  $\pi\to\prod_i \sigma_i$ is an injection.
Let $X_i=\Spec(R/\fp_i)\subset X=\Spec(R)$.      Then for any $x\in U:=X-\cup_{i\neq 1}X_i$, the surjection
	$$\pi|_x\to (\pi/\fp\pi)|_x\to\sigma|_{x}$$
is actually an isomorphism. Note that  $\Sigma\cap U$ is non-empty. As $\sigma$ is torsion-free, we have $\sigma|_x\neq0$ for any $x\in \overline{\{\eta\}}$ by \cite[Lemma 2.1.7]{EH14} and  consequently, $\pi|_x\cong \sigma|_x$ for any $x\in \Sigma\cap  \overline{\{\eta\}}$.

Up to replacing $(\pi,R)$ by $(\sigma, R/\fp)$, we may assume $R$ is an integral domain. Take an exact sequence of $k(\eta)[G(F)]$-representations
$$0\to V|_\eta\to \pi|_\eta\to W|_\eta\to0$$
with $W|_\eta$ irreducible. Let $V:=\pi\cap V|_\eta$ and $W:=\pi/V$.  Then by localization at $\eta$, the natural exact sequence
$$0\to V\to \sigma\to W\to0$$   recovers
$$0\to V|_\eta\to \pi|_\eta\to W|_\eta\to0.$$
 By the upper semicontinuous theorem,  there exists an open subset $X_K^{\prime} \subset X$ for any open compact subgroup $K\subset G(F)$ such that for any $x\in X_K^\prime$, $$\dim_{k(x)}V^K|_x=\dim_{k(\eta)}V^K|_\eta;\quad  \dim_{k(x)}\pi^K|_x=\dim_{k(\eta)}\pi^K|_\eta;\quad  \dim_{k(x)}W^K|_x=\dim_{k(\eta)}W^K|_\eta.$$
Note that $\pi|_x\cong W|_x$ for any $x\in\Sigma\cap X_K^\prime$ when $K$ is small enough, so
$$\dim_{k(\eta)}\pi^K|_\eta=\dim_{k(x)}\pi^K|_x=\dim_{k(x)} W^K|_x=\dim_{k(\eta)} W^K|_\eta.$$
  Consequently, $V|_\eta=0$ and  $\pi|_\eta$ is irreducible.

Assume $\pi|_x$ is absolutely irreducible for all $x\in\Sigma$. To show $\pi|_\eta$ is absolutely irreducible, it suffices to show $\pi|_\eta\otimes_{k(\eta)}C$ is irreducible for any finite field extension $C/k(\eta)$. Take any $R$-subalgebra $\hat{R}\subset C$  which  is a finitely generated $R$-module such that $\Frac(\hat{R})=C$ and consider $\hat{\pi}=\pi\otimes_{R}\hat{R}$.
Let $\hat{\Sigma}$ (resp. $\hat{\eta}$) be the pre-image of $\Sigma$ (resp. $\eta$ ) along  $\Spec(\hat{R})\to X$. Then  $\hat{\Sigma}$ is Zariski dense and for any $\hat{x}\in\hat{\Sigma}$, $\hat{\pi}|_{\hat{x}}$ is irreducible. Hence $\hat{\pi}|_{\hat{\eta}}=\pi|_\eta\otimes_{k(\eta)}C$ is irreducible and we are done.
 \end{proof}

For any  smooth $\CO_X[G(F)]$-module $\pi$,  $\pi^*:=\sheafhom_{\CO_X}(\pi,\CO_X)$ has an $\CO_X[G(F)]$-module structure
$$(g\cdot \ell)(v):=\ell(g^{-1}\cdot v);\quad \forall\ \ell\in\pi^*,\ v\in \pi,\ g\in G(F).$$
 The  $\CO_X[G(F)]$-submodule $\pi^\vee:=\{\ell\in\pi^*\mid\ell\ \mathrm{smooth}\}\subset \pi^*$ is called the {\em smooth dual} of $\pi$.
Let $\pi, \wt{\pi}$ be finitely generated  smooth admissible torsion-free $\CO_X[G(F)]$-modules together with a $G(F)$-equivariant pairing $\langle-,-\rangle:\ \pi\times \tilde{\pi}\to\CO_X$
such that $\pi|_x$ and $\wt{\pi}|_x$ is irreducible for each $x\in\Sigma$ and the 
specialization $\langle-,-\rangle|_x$ of $\langle-,-\rangle$ at each $x\in\Sigma$  is non-degenerate. For latter applications, we shall compare $\pi^\vee$ and $\wt{\pi}$.
\begin{lem}\label{MVWdual}Notations as above. Then $\langle -.-\rangle$ induces an injection $\wt{\pi}\hookrightarrow \pi^\vee$. In particular, $\wt{\pi}|_\eta\cong (\pi|_\eta)^\vee$ for each generic point $\eta\in X$.
\end{lem}
\begin{proof}By the Hom-Tensor adjunction, the pairing $\langle -,- \rangle$ gives an element $f\in \Hom_{\CO_X[G(F)]}(\wt{\pi}, \pi^\vee)$ such that for each $x\in\Sigma$, the induced morphism $$f|_x:\ \wt{\pi}|_x\to \pi^\vee|_x\to (\pi|_x)^\vee$$ is an isomorphism. Then  for any generic point $\eta\in X$, $$f_\eta\in \Hom_{k(\eta)[G(F)]}(\wt{\pi}_\eta, \pi^\vee_\eta)=\Hom_{\CO_X[G(F)]}(\wt{\pi},\pi^\vee)\otimes_{\CO_X}k(\eta)$$
is non-zero. Thus by Lemma \ref{irre}, $f_\eta:\ \wt{\pi}|_\eta\to \pi^\vee|_\eta$ is an isomorphism and $\Ker(f)_\eta=0$. As $\Ker(f)$ is torsion-free, one has $\Ker(f)=0$ and we are done.
\end{proof}
We remark that  if $\pi|_x$ is irreducible for all $x\in\Sigma$, one can construct a finitely generated smooth admissible torsion-free $\CO_X[G(F)]$-module $\wt{\pi}$ from $\pi$ such that for each $x\in \Sigma$,  $\wt{\pi}|_x\cong (\pi|_x)^\vee$  by the MVW involution at least for $G$  classical (see \cite{Pra19} for details). When $G=\GL_n$, one can even define a $G(F)$-invariant pairing $\langle -,-\rangle:\ \pi\times\wt{\pi}\to\CO_X$ which induces non-degenerate pairing   $\langle-,-\rangle|_x:\  \pi|_x\times\wt{\pi}|_x\to k(x)$ at each $x\in \Sigma$ upon shrinking $X$ to an open subset containing $\Sigma$ (see \cite[Proposition 5.2.6]{Dis20} for details).

\section{Local periods in families}\label{sec-conj}
In this section, we formulate several conjectures about  canonical local periods for spherical varieties. Throughout this section,  let  $G$ be a reductive group over $F$ and $Y$ be a $G$-variety  over $F$ which is
\begin{itemize}
	\item {\em homogeneous affine} i.e. $G(\bar{F})$ acts transitively on $Y(\bar{F})$ and the stabilizers of  points on $Y(\bar{F})$ are reductive.
	\item  {\em spherical} i.e. $Y$ is normal and  there is a Borel subgroup of $G_{\bar{F}}$ whose  orbit on $Y_{\bar{F}}$ is Zariski open.
\end{itemize}
Moreover, we assume $Y(F)\neq\emptyset$ and  let $H=G_{y_0}$ for a fixed $y_0\in Y(F)$. Note that $H$ is a reductive $F$-subgroup of $G$ and $Y$ is isomorphic to the geometric quotient $H\backslash G$.
\subsection{Plancherel decompositions on spherical varieties}
Fix a Haar measure  $d\mu_Y$  on $Y(F)$ such that $\int_{Y(F)}fd\mu_Y\in \BQ$  for any $f\in\CS(Y(F),\BQ)$.
Then  the Hilbert space  $L^2(Y(F))$ of square-integrable functions with respect to $\mu_Y$ is
naturally an  unitary $G(F)$-representation.  As $G(F)$ is a postliminal locally compact group,
there is  an isomorphism of unitary $G(F)$-representations, which is unique in a suitable sense
(for details, see \cite[Section 3.2]{Li18})
\[L^2(Y(F)) \stackrel{\sim}{\lra} \int_{\wh{G}} \CH_\pi d\mu(\pi).\]
Here,
\begin{itemize}
	\item $\wh{G}$ is the space of unitary irreducible $G(F)$-representations  with the
	Fell topology and $\mu$ is
	a Radon measure on $\wh{G}$.
	\item $\{\CH_\pi\}_\pi$ is a measurable field of Hilbert spaces and for each $\pi$, $\CH_\pi$ is
	a unitary representation isomorphic to a direct sum of at most countably many copies of $\pi$.
\end{itemize}	
Determining the Plancherel decomposition is a fundamental problem in harmonic analysis.

When $G=H\times H$ and $Y=\Delta(H)\bs H\times H(\cong H)$, then
\[L^2(H(F)) \cong \int_{\wh{H}} \pi \boxtimes \bar{\pi} d\mu_H^{\PL}(\pi)\]
where  $\pi \boxtimes \bar{\pi} \cong\End(\pi)$ is endowed with the Hilbert-Schmidt norm and  $\mu_H^{\PL}$
is the  Plancherel measure on $\wh{H}$. Here, $\bar{\pi}$ is the complex conjugation of $\pi$.  
Note that Harish-Chandra proved that the support of $\mu_H^{\PL}$ is the tempered representations.

When $Y$ is {\em strongly tempered} in the sense  that  the matrix coefficients of any tempered $G(F)$-representation $\pi$  are absolutely integrable over $H(F)$,  the measure $\mu$ is the Plancherel measure $\mu_G^{\PL}$  (See \cite[Theorem 6.2.1]{SV}).

 When $G$ is {\em split}, Sakellaridis-Venkatesh gives a conjectural description of the Plancherel decomposition.
Let $G^\vee$ be the dual group of $G$.  By Sakellaridis-Venkatesh \cite[Chapter 2]{SV} and Knop-Schalke \cite{KS}, there is
a dual group $G_Y^\vee$ over $\BC$ equipped with a morphism $f: G_Y^\vee \times \SL_2(\BC) \ra G^\vee$.  Denote by $G_Y$ the
split group  over $F$ dual to $G_Y^\vee$.  It is expected (see \cite[Section 17.3]{SV}) that there is a measure
$\mu_{G_Y}^{\PL}$ on the $G_Y^\vee$-conjugacy classes of tempered Langlands parameters on $G_Y$
such that the Plancherel decomposition has the form
\[L^2(G_Y(F)) \stackrel{\sim}{\lra} \int_{[\phi]} \CH_{[\phi]} d\mu_{G_Y}^{\PL}([\phi]),\quad
\CH_{[\phi]}=\bigoplus \pi\boxtimes\bar{\pi}\]
where for each $[\phi]$, $\pi$ runs over the $L$-packet associated to $[\phi]$ and the Hilbert
structure on each $\pi \boxtimes \bar{\pi}$ is given by an integer multiple of the Hilbert-Schmidt norm.

Let $W_F$ be the Weil group of $F$ and $L_F=W_F\times\SL_2(\BC)$ be the Langlands group. A (local) Arthur parameter is a homomorphism
\[\psi: L_F \times \SL_2(\BC) \lra G^\vee\]
such that $\psi\mid_{L_F}$ is a tempered (i.e. bounded on $L_F$) Langlands parameter and $\psi\mid_{\SL_2(\BC)}$ is algebraic.
Given any Arthur parameter, the associated Langlands parameter is given by
\[L_F \lra L_F \times \SL_2(\BC), \quad w \mapsto \left( w , \matrixx{|w|^{1/2}}{0}{0}{|w|^{-1/2}} \right).\]
It is conjectured that to each $G^\vee$-conjugacy class of Arthur parameter $[\psi]$, we may ``naturally''
associate a finite set of unitary $G(F)$-representations, the so-called Arthur packet of $[\psi]$.
It contains the $L$-packet of the associated Langlands parameter.

A {\em $Y$-distinguished  Arthur parameter} (see \cite[Section 16.2]{SV}) is a commutative diagram
$$\xymatrix{& G_Y^\vee \times\SL_2(\BC) \ar[dr]^{f} & \\ L_F\times \SL_2(\BC) \ar[rr]^{\psi} \ar[ur]^{(\phi,\id)} & & G^\vee}$$
where $\psi:\ L_F \times \SL_2(\BC) \lra G^\vee$ ia an Arthur parameter and
$\phi:\ L_F  \lra G_Y^\vee$ is a tempered Langlands parameter.
The local conjecture  predicts the Plancherel decomposition of $L^2(Y(F))$ is given by the
$Y$-distinguished Arthur parameters of $G$.
\begin{conj}[Weak form of the local conjecture {\cite[Conjecture 16.2.2]{SV}}] \label{conj-local}
	There is a direct integral decomposition
	\[\tag{$*$} L^2(Y(F)) \stackrel{\sim}{\lra} \int_{[\psi]} \CH_{[\psi]} \mu_{G_Y}^{\PL}([\psi])\]
	where
	\begin{itemize}
		\item $[\psi]$ runs over $G_Y^\vee$-conjugacy classes of $Y$-distinguished Arthur parameter.
		\item $\CH_{[\psi]}$ is isomorphic to a (possibly empty) direct sum of irreducible representations
		belonging to the Arthur packet of $[\psi]$.
	\end{itemize}
\end{conj}
\begin{remark}In general, Conjecture \ref{conj-local} states only necessary conditions for a representation to appear in the Plancherel decomposition of $L^2(Y(F))$. When $Y$ is strongly tempered and $G$ is split, it is claimed that $G_Y^\vee = G^\vee$  in \cite[Page 7]{SV} and the Plancherel decomposition in \cite[Theorem 6.2.1]{SV} is compatible with the form in Conjecture \ref{conj-local}.
\end{remark}

\subsection{Canonical local periods} We now  construct a ``canonical'' bi-$H(F)$-invariant pairing
on $\pi\otimes \pi^\vee$  for $\mu_{G_Y}^{\PL}$-almost all  $\pi$ in the Plancherel decomposition of $L^2(Y(F))$.

Let $C^\infty(Y(F))$ (resp. $\CS(Y(F))$) be the space of complex valued smooth (resp. Schwartz) functions on $Y(F)$ equipped with the $G(F)$-action given by right translation. 
The $G(F)$-invariant bilinear pairing
\[\CS(Y(F))\times C^\infty(Y(F))\to\BC,\quad (f,g)\mapsto\int_{Y(F)}f(y)g(y)d\mu_Y(y)\]
induces an injection of $G(F)$-modules
\[\CS(Y(F))^\vee \hookrightarrow C^\infty(Y(F))\]
identifying $\CS(Y(F)^\vee$ with vectors in $C^\infty(Y(F))$ which are smooth under the action of $G(F)$.

For any irreducible smooth admissible  $G(F)$-representation $\pi$,
denote by
$$\CS(Y(F))_\pi:=\CS(Y(F))/\CN_\pi,\quad \CN_\pi:=\bigcap_{f\in \Hom_{G(F)}(\CS(Y(F)),\pi)}\ker f$$
 the maximal $\pi$-isotypic quotient of $\CS(Y(F))$.

Assume now $G$ is split and  Conjecture \ref{conj-local} holds. Then by the Gelfand-Kostyuchenko method (see \cite{Li18}), the decomposition $(*)$
determines a $G(F)$-invariant positive semi-definite Hermitian form
\[ \langle \cdot,\cdot \rangle_\pi:\ \CS(Y(F))\times \CS(Y(F)) \lra \BC\]
which factors through  $\CS(Y(F))_\pi$ for $\mu_Y^{\PL}$-almost all $\pi$ such that

\[\int_{Y(F)} f_1(x)\ov{f_2(x)} dx  = \int_{[\psi]} \sum_{\pi \in [\psi]}
\langle f_1,f_2 \rangle_{\pi} d\mu_{G_Y}^{\PL}([\psi]),\quad \forall f_1,f_2 \in \CS(Y(F)) .\]
It is expected (see \cite[Section 17.3]{SV}) and we will assume that $\langle \cdot,\cdot\rangle_\pi$
is well-defined for \textit{all} the unitary representation $\pi$ appearing in the decomposition $(*)$.

Note that by \cite[Section 6.1]{SV}, there is a canonical isomorphism
\[\CS(Y(F))_\pi \cong V(\pi) \otimes \pi,\quad V(\pi):=\Hom_{G(F)}(\CS(Y(F)),\pi)^*.\]
Then the complex conjugation $\ov{\CS(Y(F))_\pi}$ of $\CS(Y(F))_\pi$ is isomorphic to  $\CS(Y(F))_{\bar{\pi}}$
canonically and the Hermitian form $\langle\cdot,\cdot\rangle_\pi$ on $\CS(Y(F))_\pi$ induces a non-trivial  $G(F)$-invariant
linear functional
\[\ell_\pi:\ \CS(Y(F))_\pi \otimes \CS(Y(F))_{\bar{\pi}}\to\BC.\]
Note that  the irreducibility of $\pi$ implies
$$\dim_{\BC}\Hom_{G^2(F)}(\pi\boxtimes\bar{\pi},\pi\boxtimes\bar{\pi})=\dim_{\BC}\Hom_{G(F)}(\pi\otimes\bar{\pi},\BC)=1.$$
The following lemma is straightforward.
\begin{lem}Any non-degenerate $G(F)$-invariant linear functional $\ell$ on $\pi \otimes \bar{\pi}$ induces an isomorphism
	\begin{align*}
	\Hom_{G^2(F)}\left( \CS(Y(F))_\pi \boxtimes \CS(Y(F))_{\bar{\pi}}, \pi \boxtimes \bar{\pi} \right)
	& \stackrel{\sim}{\lra} \Hom_{G(F)}\left( \CS(Y(F))_\pi \otimes \CS(Y(F))_{\bar{\pi}},\BC\right)\\
	p&\longmapsto \ell\circ{p}
	\end{align*}
\end{lem}
\begin{proof}
By  \cite[Theorem 3.1]{Pra18}, one has	\begin{align*}\Hom_{G(F)^2}\left( \CS(Y(F))_\pi \boxtimes \CS(Y(F))_{\bar{\pi}},
	\pi \boxtimes \bar{\pi} \right) &\cong
	\Hom(V(\pi) \otimes V(\bar{\pi}),\BC) \otimes \Hom_{G^2(F)}(\pi \boxtimes \bar{\pi},
	\pi \boxtimes \bar{\pi})\\
	&\cong \Hom(V(\pi) \otimes V(\bar{\pi}),\BC) \otimes \Hom_{G(F)}(\pi \otimes \bar{\pi},\BC)\\
	&=\Hom_{G(F)}( (V(\pi) \otimes V(\bar{\pi})) \otimes (\pi \otimes \bar{\pi}),\BC)\\
	&\cong 	\Hom_{G(F)}(\CS(Y(F))_\pi \otimes \CS(Y(F))_{\bar{\pi}},\BC)
	\end{align*}
	where the second $\cong$ is given by composing with  $\ell$.
\end{proof}
 Fix a non-degenerate $G(F)$-invariant pairing $\ell$ on $\bar{\pi}\otimes \pi$. Let
\[p_\ell:\ \CS(Y(F)) \boxtimes \CS(Y(F)) \lra \CS(Y(F))_{\bar{\pi}}\boxtimes \CS(Y(F))_\pi \lra \bar{\pi} \boxtimes \pi \]
be the unique $G^2(F)$-equivariant map such that $\ell\circ p_{\ell}=\ell_{\bar{\pi}}$
by the above lemma.
Taking dual, we obtain a $G^2(F)$-equivariant map
$$p_{\ell}^\vee:
\pi \boxtimes \bar{\pi} \stackrel{\sim}{\lra} (\bar{\pi}\boxtimes\pi)^\vee\to (\CS(Y(F))^2)^\vee
\stackrel{}{\lra} C^{\infty}(Y(F))^2.$$
Composing with the $H^2(F)$-equivariant evaluation map
$$\mathrm{ev}_{y_0,y_0}:\ C^{\infty}(Y(F))^2\to \BC,\quad f\mapsto f(y_0,y_0),$$ one obtains a bi-$H(F)$-invariant linear functional
\[P_{\pi,\ell}: \pi \boxtimes \bar{\pi} \xrightarrow{p_{\ell}^\vee} C^\infty(Y(F))^2 \stackrel{\mathrm{ev}_{y_0,y_0}}{\lra} \BC.\]

View $\ell$ as a  $G(F)$-invariant pairing on $\pi\otimes \bar{\pi}$ via the natural identification $$\bar{\pi}\otimes \pi=\pi\otimes\bar{\pi}\quad  a\otimes b\mapsto b\otimes a.$$ By construction, it is straightforward to check (see \ref{fibers} for notations)
\begin{itemize}
	\item The $\BC^\times$-subset $$(\pi\otimes\bar{\pi})^{\star,-1}=\{ v \in \pi\otimes\bar{\pi}\mid \ell(v)\neq0 \}$$ is independent of the choice of $\ell$.
	\item  Denote the space  $Q_{(\pi\otimes\bar{\pi})^{\star}}(\pi\otimes\bar{\pi})$  by $Q(\pi\otimes\bar{\pi})$. Then the map
	\[Q_{\pi\otimes\bar{\pi}}:\  Q(\pi\otimes\bar{\pi})\to\BC,\quad \frac{v}{w}\mapsto  \frac{P_{\pi,\ell}(v)}{\ell(w)}\]
	is independent on the choice of $\ell$.
\end{itemize}

Note that $\bar{\pi}\cong \pi^\vee$ as $G(F)$-representations.
\begin{defn}\label{Can Period}Let $Q_\pi:\ Q(\pi\otimes\pi^\vee)\to\BC$ be the map $Q_{\pi\otimes\bar{\pi}}$ precomposed with any isomorphism $\kappa:\ \pi^\vee\cong\bar{\pi}$.
\end{defn}
By the above discussions, we have the following lemma.
\begin{lem}The map $Q_\pi$ is independent of the choice of the isomorphism in Plancherel decomposition
	$(*)$, the $G(F)$-invariant linear form $\ell$  and the isomorphism $\kappa$.
\end{lem}

When the spherical variety $Y$ is  strongly tempered, we can proceed in a same manner to define  $Q_\pi$ by taking the measure $\mu_G^{\PL}$ in the Plancherel decomposition.  Moreover, in this case, there is a down-to-earth description of $Q_\pi$.
\begin{prop}[Theorem 6.2.1 \cite{SV}] \label{SRP}
	Assume $Y$ is a strongly tempered spherical $G$-variety. Then for any tempered $G(F)$-representation $\pi$,
	\[Q_\pi\left(\frac{v}{w}\right) =
	\frac{\int_{H(F)}\ell(\pi(h)v)dh}{\ell(w)}, \quad v,w\in\pi \otimes \pi^\vee\quad \ell(w) \not= 0.\]
	Here,  $\ell:\ \pi \otimes \pi^\vee\to\BC$ is any non-trivial $G(F)$-invariant linear form  and the Haar measure on $H(F)$ is
	determined by those measures on $Y(F)$ and $G(F)$.
\end{prop}
 \subsection{Families of local periods}
We start with the rationality property of $Q_\pi$. For a smooth admissible representation $\pi$ over $E$, let $\CE(\pi)$ be the set of field embeddings $\tau:\ E\to\BC$ such that $\pi_\tau$ is irreducible and   appears in the Plancherel decomposition of $L^2(Y(F))$.
\begin{conj}[Rationality $Q$]\label{Rationality Q}  Let $\pi$ be a smooth admissible $G(F)$-representation  over $E$ with non-empty $\CE(\pi)$. 	Let $Q(\pi\otimes_E\pi^\vee):=Q_{(\pi\otimes_E\pi^\vee)^\star}(\pi\otimes_E\pi^\vee)$ where  $$(\pi\otimes_E\pi^\vee)^\star=\{\sum_i v_i\otimes v^i\in\pi\otimes\pi^\vee\mid \sum v^i(v_i)\neq0\}$$
	Then there exists a unique $E$-linear map $$Q_\pi:\
	Q(\pi\otimes_E\pi^\vee)\to E$$
	such that for any $\tau\in\CE(\pi)$,  the following diagram commutes
	$$\xymatrix{ Q(\pi\otimes_E\pi^\vee) \ar[d]^{\tau} \ar[r]^{\quad   Q_\pi} & E \ar[d]^{\tau} \\
		Q(\pi_\tau\otimes_{\BC}\pi_\tau^\vee) \ar[r]^{\quad Q_{\pi_\tau}}  & \BC.}$$
\end{conj}

Assuming the rationality conjecture, we now discuss the behavior of $Q_\pi$  when $\pi$ varies in families.
\begin{conj}[Meromorphy Q]\label{local-conj} Let $\pi$, $\wt{\pi}$ be  finitely generated smooth admissible  torsion-free $\CO_X[G(F)]$-modules such that  for any $x\in\Sigma$,  $\CE(\pi|_x)\neq\emptyset$ and $\wt{\pi}|_x \cong (\pi|_x)^\vee$.
Let
 $(\pi \otimes \wt{\pi})^{\star}\subset \pi\otimes\wt{\pi}$ be the subsheaf such that
	for any open subset $U\subset X$,
	$$(\pi \otimes \wt{\pi})^{\star}(U)=\{s\in (\pi\otimes\wt{\pi})(U)\mid s(x)\in (\pi\otimes\wt{\pi})|_x^{\star},\quad \forall\ x\in \Sigma\cap U\}$$
	and set $Q(\pi\otimes\wt{\pi}):=Q_{(\pi\otimes\wt{\pi})^\star}(\pi\otimes\wt{\pi})$.
	Then the family $\{Q_{\pi|_x}:\ Q(\pi\otimes\wt{\pi})|_x\to k(x)\}_{x\in\Sigma}$ is meromorphic.
\end{conj}
\section{Strongly tempered spherical varieties without type $N$-roots}\label{type N}

In this section, we  prove  Conjectures \ref{conj-2} for strongly tempered spherical varieties without type $N$-roots using the admissibility of Jacquet modules and the asymptotic behaviour of matrix coefficients. We shall assume $X = \Spec(R)$ where
$R$ is a reduced $E$-algebra of finite type.

We need some preliminary results. 

Firstly, we record the  admissibility of Jacquet modules established in \cite[Corollary 1.5]{DHK22}.
\begin{thm}\label{admjac}Let $P=MN\subset G$ be any parabolic subgroup with Levi factor $M$. Then  for any smooth admissible 
$R[G(F)]$-module $\pi$, the Jacquet module $$J_N(\pi):=\pi/ \pi(N),\quad \pi(N):=\langle\pi(n)v-v|v\in\pi, n\in N(F)\rangle\subset \pi$$
is a smooth admissible $R[M(F)]$-module.
\end{thm}

Secondly, we need the  following  consequence of Casselman's canonical pairing for $G(F)$-representations.   Fix a minimal parabolic subgroup $P_{\emptyset_G}=M_{\emptyset_G} N_{\emptyset_G}\subset G$. Let $A_{\emptyset_G}\subset M_{\emptyset_G}$ be a maximal split torus  and $\Delta_G$ be the set of simple roots with respect to $(P_{\emptyset_G}, A_{\emptyset_G})$.  Let
\[A_{\emptyset_G}^-:=\{a\in A_{\emptyset_G}(F)|\ |\alpha(a)|\leq 1\ \forall\ \alpha\in \Delta_G\}.\]
Now we record Casselman's canonical pairing theorem \cite[Theorem 4.3.3]{Cas95}:
\begin{thm}\label{cano}Let $\pi$ be a smooth admissible $G(F)$-representation over a characteristic zero field $K$. Then for any $\Theta_G\subset \Delta_G$, there exists a $M_{\Theta_G}$-invariant non-degenerate pairing 
$$(-,-)_{N_{\Theta_G}}:\ J_{N_{\Theta_G}}(\pi)\times J_{N^-_{\Theta_G}}(\pi^\vee)\to K $$
and a constant $\epsilon>0$ such that  for any $a\in A_{\emptyset_G}^-$ with $|\alpha(a)|<\epsilon$ for all $\alpha\in \Delta_G-\Theta_G$, $$ (\pi(a)v,v^\vee)=(\pi(a)v,v^\vee)_{N_{\Theta_G}},\quad \forall\ v\in\pi, v^\vee\in\pi^\vee.$$
Here $(-,-)$ is the contraction on $\pi\times\pi^\vee$, $N_{\Theta_G}^-$ is the unipotent subgroup opposite to $N_{\Theta_G}$  and in the right hand side, $\pi(a)v$ and $v^\vee$ are viewed as elements in $J_{N_{\Theta_G}}(\pi)$ and $J_{N^-_{\Theta_G}}(\pi)$ respectively.
\end{thm}
The following immediate consequence is crucial for our strategy:
\begin{cor}\label{Asym}
Let $\pi$ be a finitely generated smooth admissible $R[G(F)]$-module satisfying the assumptions in Conjecture \ref{conj-2}. Then for   any $\Theta_G\subset \Delta_G$,  any $v\in\pi(N_{\Theta_G})$ and any $\wt{v}\in \wt{\pi}$, there exists $\epsilon>0$ such that  for any $a\in A_{\emptyset_G}^-$ with $|\alpha(a)|<\epsilon$ for all $\alpha\in \Delta_G-\Theta_G$, $\langle \pi(a)v,\tilde{v} \rangle=0$.
\end{cor}
\begin{proof}By Lemma \ref{irre} and Lemma \ref{MVWdual},
it suffices to prove the statement for the contraction
$$\pi|_\eta\times (\pi|_\eta)^\vee \to k(\eta)$$
for the finitely many generic points $\eta\in X=\Spec(R)$. This immediately follows from  Theorem \ref{cano}.
\end{proof}

Thirdly, we need the Cartan decomposition for $H$. Let $P_{\emptyset_H}=M_{\emptyset_H} N_{\emptyset_H}\subset H$ be a minimal parabolic subgroup. Let $A_{\emptyset_H}\subset M_{\emptyset_H}$ be a maximal split torus  and $\Delta_H$ be the set of simple roots with respect to $(P_{\emptyset_H}, A_{\emptyset_H})$.  For any  $\Theta_H\subset\Delta_H$, denote by
$$ A_{\Theta_H}:=\{x\in A_{\emptyset_H}|\ \alpha(x)=1\quad \forall\ \alpha\in \Theta_H\}.$$
Set  $\CT_{\Theta_H}^{?}:= A_{\Theta_H}^{?}/A_{\Theta_H}^\circ$ for $?=-$,$--$, where
$$A_{\Theta_H}^-:=\{x\in A_{\Theta_H}(F)|\  |\alpha(x)|\leq 1 \quad \forall\ \alpha\in \Delta_H-\Theta_H\};$$
 $$A_{\Theta_H}^{--}:=\{x\in A_{\Theta_H}(F)|\  |\alpha(x)|< 1 \quad \forall\ \alpha\in \Delta_H-\Theta_H\};$$
 $$A_{\Theta_H}^\circ:=\bigcap_{\alpha \in \Rat(A_{\Theta_H})}\ker|\alpha|=A_{\Theta_H}(\CO_F).$$
 Here, $\Rat(A_{\Theta_H})$ denotes the set of rational characters on $A_{\Theta_H}$.

\begin{thm}[Cartan Decomposition]\label{Cartan}There exists a maximal compact open subgroup $K_H\subset H(F)$, a finite subset $\omega\subset M_{\emptyset_H}$ normalizing $K_H$ such that $$H(F)=\bigsqcup_{w\in \omega, t\in\CT_{\emptyset_H}^-}K_H w t K_H.$$
Moreover, if the center $Z_H$ of $H$ is anisotropic or $H$ is unramified, then one can choose $\omega=\{e\}$.
\end{thm}
\begin{proof}See \cite[Section 4.4]{BT72}. For the case $Z_H$ is isotropic,  see also \cite[Lemma 1.4.5]{Cas95}. The case $H$  unramified is also explained in  \cite[Section 1]{Cas12}.
    \end{proof}
Finally, we record the following volume formula.
\begin{prop}\label{Vol} For any $\Theta_H\subset \Delta_H$, there exists a constant $C_{\Theta_H}\in\BQ^\times$
	  such that
   \[\Vol(K_HtK_H)=C_{\Theta_H}\delta_{P_{\emptyset_H}}^{-1}(t),\quad \forall \ t\in \CT_{\Theta_H}^{--}.\]
	Here $\delta_{P_{\emptyset_H}}$ is the modulus character of $P_{\emptyset_H}$.
\end{prop}
\begin{proof}When $H$ is unramified, see \cite[Proposition 1.6]{Cas12}. In general, one can deduce the result from \cite[Theorem 4.2]{Cas80}.
\end{proof}

Now we  introduce the  notion of \dfn{admitting a reduction structure} for spherical varieties,  which is crucial for
Theorem \ref{main}. The reduction structure deals with the regions $\CT_{\Theta_H}^{--}$, $\Theta_H\subset\Delta_H$. Recall the following facts for $?=-,--$ and  $\Theta_H\subset \Delta_H$: 
\begin{itemize}
\item  $\CT_{\Theta_H}^-$ is naturally an abelian semigroup. Moreover any $t\in T_{\Theta_H}^-$ induces an injection
	\[T_t: \CT_{\Theta_H}^? \hookrightarrow \CT_{\Theta_H}^?\]
    \item 
the natural map $\CT^?_{\Theta_H^\prime}\to \CT^?_{\Theta_H}$ is an injection for any $\Theta_H\subset \Theta_H^\prime\subset\Delta_H$. Moreover,
     \[\CT_{\Theta_H}^-=\bigsqcup_{\Theta_H\subset \Theta_H^\prime}\CT_{\Theta_H^\prime}^{--}.\]
\end{itemize}
 For each region $\CT_{\Theta_H}^{--}$, the set  $ \Delta_H-\Theta_H$  of simple roots induces the map:
 $$\CT_{\Theta_H}^{--}\lra \BZ_{\geq1}^{\Delta_H-\Theta_H}, \quad t\mapsto \ \big(v_\alpha(t):=\val(\alpha(t))\big)_{\alpha\in \Delta_H-\Theta_H}.$$
To control the kernel,  we 
fix a finite subset $C_H\subset \Rat(A_{\emptyset_H})$  such  that $\Delta_H\cap C_H=\emptyset$ while
$\Delta_H\cup C_H$ is a basis of
$\Rat(A_{\emptyset_H})\otimes_{\BZ}\BQ$. Note that the cardinality of $C_H$ is just the split rank of the center $Z_H$ of $H$ and one has the embedding
\begin{equation}\label{embedding-CT}
\xymatrix{& \CT_{\Theta_H}^{--} \ar@{^(-_>}[r] & \BZ_{\geq1}^{\Delta_H-\Theta_H}\times \BZ^{C_H} ,\quad 
t\mapsto \left((v_\alpha(t)_{\alpha\in \Delta_H-\Theta_H}, (v_\beta(t))_{\beta\in C_H} \right).}
\end{equation}
In fact, for the pairs $(G,H)$ we shall consider, the split rank of $Z_H$ is at most $1$ and
$C_H$ is either the empty set or a singleton. 
For the convenience of readers, we now  introduce the reduction structure for these two cases separately. 

For now on, assume $A_{\emptyset_H} \subset A_{\emptyset_G}$.  Firstly, we consider the case $C_H=\emptyset$.
\begin{defn}[Case $C_H=\emptyset$]\label{reduction structure I}
  A \dfn{reduction structure} for $\Theta_H\subset\Delta_H$  with respect to $G$ is a finite set
\[S = S(\Theta_H) = \{(\Theta,w, s(\Theta,w))\}\]
where
\begin{itemize}
\item $\Theta\subset \Delta_G$,
\item $w\in W_G:=Z_G(A_{\emptyset_G})\bs N_G(A_{\emptyset_G})$, the
Weyl group of $A_{\emptyset_G}$,
\item $s=s(\Theta,w)\in w^{-1}A_{\Theta}^-w\cap A_{\emptyset_H}^-$
whose image  in $\CT_{\emptyset_H}^{-}$ belongs to $\CT_{\Theta_H}^{-}-\{e\}$
       \end{itemize}
satisfying the following two finite conditions
\begin{enumerate}[(F1)]
        \item for any $(\Theta,w,s)\in S$  and any $n\in\BN$,
        $\CT_{\Theta_H}^{--}-T_{s^n}(\CT_{\Theta_H}^{--})$ is a finite disjoint union of  sets of the form $T_t\CT_{\Theta_H^\prime}^{--}$ with $t \in \CT_{\Theta_H}^{-}$ and $\Theta_H\subsetneq \Theta_H^\prime$. (When $\Theta_H=\Delta_H$,  $\CT_{\Theta_H}^{--} = T_{s^n}(\CT_{\Theta_H}^{--})$ automatically);
    \item  for any $0<\epsilon\leq 1$, the image of
$$\bigcap_{(\Theta,w,s)\in S}\left\{a\in A_{\Theta_H}^{--} \Big| \ \mathrm{if}\ waw^{-1}\in A_{\emptyset_G}^-,\ \exists \alpha\in \Delta_G-\Theta, |\alpha(waw^{-1})|\geq \epsilon\right\}$$
	    in  $\CT_{\Theta_H}^{--}$ is finite, while the image is not finite when $S$ is replaced by  any proper subset.
\end{enumerate}
Here if $S=\emptyset$, (F1) is empty and (F2) reads that 
$\CT_{\Theta_H}^{--}$ is finite (in fact, a singleton).

\end{defn}
The reduction structure is introduced to analyze the rationality of the following formal power series.
Let $\pi$ be a finitely generated smooth admissible $R[G(F)]$-module satisfying the assumptions in
Conjecture \ref{conj-2}. In particular, there is a $R[G(F)]$-module $\wt{\pi}$ with an invariant $R$-bilinear
pairing $\langle \cdot,\cdot \rangle$ on $\pi \times \wt{\pi}$.
 For any subset $\Theta_H \subset \Delta_H$, $v \in \pi^{A_{\emptyset_H}^\circ}$ and $\wt{v} \in \wt{\pi}^{A_{\emptyset_H}^\circ}$, consider
\[\wt{F}_{\Theta_H,v,\wt{v}}(T):=\sum_{t\in \CT_{\Theta_H}^{--}}
\langle \pi(t)v, \wt{v}\rangle T^{v(t)}, \quad T^{v(t)}:=\prod_{\alpha\in \Delta_H-\Theta_H} T_\alpha^{v_{\alpha}(t)}.\]
\begin{example}[$H$-anisotropic case]\label{aniso}
When $H$ is anisotropic, then $\Delta_H=C_H=\emptyset$. In this case,  $\Delta_H$ admits the reduction structure $\emptyset$ and  $\wt{F}_{\Delta_H,v,\wt{v}}=\langle v, \wt{v}\rangle$ is constant.
\end{example}
\begin{example}[The triple product case]\label{Triple}
Let $H = \PGL_2$ embed into $G = \BG_m \bs \GL_2^3$ diagonally. One has $C_H=\emptyset$ and $\Delta_H = \{\alpha\}$ with
$\alpha\left[ \matrixx{t_1}{0}{0}{t_2} \right] =  t_1/t_2$.
Then  $\wt{F}_{\Delta_H,v,\wt{v}}=\langle v,\wt{v}\rangle$ and 
\[\wt{F}_{\emptyset_H,v,\wt{v}}(T) = \sum_{a \geq 1} \langle \pi(t(a,0))v,\wt{v} \rangle T^a, \quad t(a,0) =
\matrixx{\varpi^a}{0}{0}{1}.\]
  We now discuss the rationality of $\wt{F}_{\emptyset_H,v,\wt{v}}(T)$. With respect to $G$,  $\emptyset_H$ admits the reduction structure
\[S = \{(\emptyset_G,e,s=t(1,0))\}\]
since
\begin{enumerate}
\item[(F1)] for each $n$, we have the {\em finite} decomposition
\[\CT_{\emptyset_H}^{--}-T_{s^n}(\CT_{\emptyset_H}^{--})=\bigsqcup_{i=1}^{n}T_{s^i}(\CT_{\Delta_H}^{--})\]
\item[(F2)] for any $0 < \epsilon \leq 1$,
\[\left\{ a \in A_{\emptyset_H}^{--} \Big| |\alpha(a)| > \epsilon \right\}/A_{\emptyset_H}^\circ\]
is  finite.
\end{enumerate}
By the admissibility of Jacquet modules, there is a non-zero polynomial
$P(X)=\sum_{n\geq0}c_n X^n$  such that $P(\pi(s))v\in \pi(N_{\emptyset_G})$ for all $v\in\pi$.
By (F1), one has  the equation of formal power series
$$ \wt{F}_{\emptyset_H, P(\pi(s))v,\wt{v}}(T)=P(T^{-1})\wt{F}_{\emptyset_H,v,\wt{v}}(T)-\sum_{n\geq1}\sum_{a=1}^nc_n\langle
\pi(t(a,0))v,\wt{v}\rangle T^{a-n}.$$
Corollary \ref{Asym} and (F2) imply the rationality of $\wt{F}_{\emptyset_H,P(\pi(s))v,\wt{v}}$. The
 rationality of $\wt{F}_{\emptyset_H,v,\wt{v}}$ then follows.
\end{example}

Now we turn to the case $C_H$ is a singleton, say $C_H = \{\beta\}$. 
For each $\Theta_H \subset \Delta_H$, consider the 
domain $\BZ_{\geq1}^{\Delta_H-\Theta_H}\times \BZ$ embedding \ref{embedding-CT}. All the coordinates are contained 
in the ``cone'' $\BZ_{\geq 1}$ except the last one which corresponds to $\beta$. To define the right object $\wt{F}_{\Theta_H,v,\wt{v}}$ (so it is a formal power series), we  further decompose  $\CT_{\Theta_H}^{?}$ into disjoint unions of ``cones'':
$$\CT_{\Theta_H}^{?}=\CT_{\Theta_H}^{?,+}\bigsqcup \CT_{\Theta_H}^{?,0}\bigsqcup \CT_{\Theta_H}^{?,-},\quad ?=-,--$$
where for $*=+,0,-$, $\CT_{\Theta_H}^{?,*}\subset\CT_{\Theta_H}^?$ is the subset cut out by the condition $v_{\beta}(x)>0, =0, <0$  respectively.


For $*=+,0,-$, we have 
\[\CT_{\Theta_H}^{-,*}=\bigsqcup_{\Theta_H\subset \Theta_H^\prime }\CT_{\Theta_H^\prime}^{--,*}\]
 and for any $t\in \CT_{\Theta_H}^{-,*}$, 
 \[T_t(\CT_{\Theta_H}^{?,*})\subset \CT_{\Theta_H}^{?,*}.\]

Let $A_{\Theta_H}^{?,*}$ be the pre-image of $\CT_{\Theta_H}^{?,*}$ in $A_{\Theta_H}^?$.

\begin{defn}[Case $\sharp C_H=1$]\label{reduction structure II}
For any $\Theta_H\subset \Delta_H$ and $*=+,0,-$, a  \dfn{reduction structure}  with respect to $G$ for the pair $(\Theta_H, *)$ is a  finite set
\[S = S(\Theta_H, *) = \{(\Theta,w, s(\Theta,w))\}\]
where
\begin{itemize}
\item $\Theta\subset \Delta_G$ and $w\in W_G$;
\item $s=s(\Theta,w)\in w^{-1}A_{\Theta}^-w\cap A_{\emptyset_H}^{-}$
      whose image in $\CT_{\emptyset_H}^{-}$ belongs to $\CT_{\Theta_H}^{-,0}\bigcup \CT_{\Theta_H}^{-,*}-\{e\}$
       \end{itemize}
satisfying the following two finite conditions
\begin{enumerate}[(F1)]
        \item for any $(\Theta,w,s)\in S$  and  $n\in\BN$,
        $\CT_{\Theta_H}^{--,*}-T_{s^n}(\CT_{\Theta_H}^{--,*})$ is a finite disjoint union of  sets of the form \begin{itemize}
            \item $T_t\CT_{\Theta_H^\prime}^{--,*}$ with $t \in \CT_{\Theta_H}^{-,*}$ and $\Theta_H\subsetneq \Theta_H^\prime$   if $*=0$;
            \item $T_t\CT_{\Theta_H^\prime}^{--,*}$ with 
            $t \in \CT_{\Theta_H}^{-,*}$ and $\Theta_H\subsetneq \Theta_H^\prime$ or 
            $T_{t_1}\CT_{\Theta_H}^{--,0}$ with $t_1 \in \CT_{\Theta_H}^{-,*}$   if $*=+,-$; (both types of sets
             may occur in the union)
        \end{itemize}
    \item  for any $0<\epsilon\leq 1$, the image of
$$\bigcap_{(\Theta,w,s)\in S}\left\{a\in A_{\Theta_H}^{--,*} \Big| \ \mathrm{if}\ waw^{-1}\in A_{\emptyset_G}^-,\ \exists \alpha
\in \Delta_G-\Theta, |\alpha(waw^{-1})|\geq \epsilon\right\}$$
	    in  $\CT_{\Theta_H}^{--,*}$ is finite, while the image is not finite when $S$ is replaced by  any proper subset.
\end{enumerate}
Here if $S=\emptyset$, (F1) is empty  and (F2) means $\CT_{\Theta_H}^{--,*}$ is finite. 

For $\Theta_H$, if $S(\Theta_H,*)$ is a reduction structure for $(\Theta_H,*)$, then we call
\[S = S(\Theta_H) =  \bigsqcup_{*=+,0,-} S(\Theta_H,*)\]
a {\em reduction structure for $\Theta_H$}. 


\end{defn}
These reduction structures are introduced to analyze the following  formal power series: 
\[\wt{F}_{\Theta_H,v,\wt{v}}(T):=\sum_{*=+,0,-} \wt{F}_{\Theta_H,v,\wt{v}}^{*}(T),\quad v\in \pi^{A_{\emptyset_H}^\circ}, \wt{v}\in\wt{\pi}^{A_{\emptyset_H}^\circ}\] 
\[\wt{F}_{\Theta_H,v,\wt{v}}^{*}(T):=\sum_{t\in \CT_{\Theta_H}^{--,*}}
\langle \pi(t)v, \tilde{v}\rangle T^{|v(t)|}, \quad T^{|v(t)|}:=T_{\beta}^{|v_\beta(t)|}
\prod_{\alpha\in\Delta_H-\Theta_H} T_\alpha^{v_\alpha(t)}.\]
Here we put $|\cdot|$ on the power of $T_\beta$ (make difference only when $*=-$) to  obtain formal power series.  

\begin{example}[The case $\BG_m\subset \PGL_2$]\label{Wal}
Consider the pair
$$H=\BG_m\hookrightarrow G=\PGL_2,\quad t\mapsto  \begin{pmatrix} t & 0\\ 0 & 1\end{pmatrix}$$
and choose $C_H=\{\id\}$. Then $\wt{F}_{\emptyset_H,v,\wt{v}}^0=\langle v, \wt{v}\rangle$ and
\begin{align*}
    \wt{F}_{\emptyset_H, v, \wt{v}}^{\pm}(T)=\sum_{a\geq1}\langle \pi(\varpi^{\pm a})v,\wt{v}\rangle T^{a}
\end{align*}
The pairs $(\emptyset_H,\pm)$ have the reduction structures 
\[ S(\emptyset_H,+)=\left\{(\emptyset_G,e,\varpi)\right\},\quad S(\emptyset_H,-)=\left\{(\emptyset_G,w,\varpi^{-1})\right\},\ w:=\begin{pmatrix} 0 & 1\\ 1 & 0\end{pmatrix}
\]
since 
\begin{itemize}
    \item for any $0<\epsilon\leq1$, $\{a\in F^\times|\epsilon\leq |a|<1\}/\CO_F^\times$ and $\{a\in F^\times|1< |a|\leq\epsilon^{-1}\}/\CO_F^\times$ are finite;
    \item  for $1\leq n\in\BN$, $$\CT_{\emptyset_H}^{--,\pm}-T_{\varpi^{\pm n}}\CT_{\emptyset_H}^{--,\pm}=\bigsqcup_{i=1}^{n}T_{\varpi^{\pm i}}\CT_{\emptyset_H}^{--,0}.$$
\end{itemize}
Take any non-zero polynomial
 $P(X)$ (resp. $P_w(X)$)  such that  for all $v\in \pi$, $P(\pi(\varpi))v\in \pi(N_{\emptyset_G})$
(resp.  $P_{w}(\pi(\varpi^{-1}))v\in \pi(wN_{\emptyset_G}w^{-1})$).
Then  
$$\wt{F}^+_{\emptyset_H, P(\pi(\varpi))v,\wt{v}}(T)=P(T^{-1})\wt{F}^+_{\emptyset_H, v, \wt{v}}(T)+G_+(T^{-1}),$$
$$\wt{F}^-_{\emptyset_H,P_w(\pi(\varpi^{-1}))v,\wt{v}}(T)=P_w(T^{-1})\wt{F}^-_{\emptyset_H, v, \wt{v}}(T)+G_{-}(T^{-1})$$
for some   polynomials $G_{+}$ and $G_{-}$. By Corollary \ref{Asym} and (F2), the summations in the left hand are actually finite.  
The rationality of $\tilde{F}^\pm_{\emptyset_H, v, \wt{v}}(T)$ then follows.
\end{example}
\begin{remark} Actually the above decomposition can
be generalized to the case $\sharp C_H\geq2$ so that the following proposition  still holds.
Precisely, one can decompose 
$$\CT_{\Theta_H}^{?}=\bigsqcup_{(I,\epsilon_I)}\CT_{\Theta_H,I}^{?,\epsilon_I},\quad I\subset C_H,\  \epsilon_{I}=(\epsilon_\beta)\in\{\pm\}^{C_H-I}$$ and for $?=-,--$, 
$\CT^{?,\epsilon_{I}}_{\Theta_H,I}:=A_{\Theta_H,I}^{?,\epsilon_{I}}/A_{\Theta_H}^{\circ}$ where
$$A^{?,\epsilon_{I}}_{\Theta_H,I}:=\{x\in A_{\Theta_H}^?|\  v_{\beta}(x)=0, \ \forall\ \beta\in I,\ \epsilon_\beta v_\beta(x)>0,\ \forall\ \beta\in C_H-I\}.$$
Then we can define a reduction structure for the triple $(\Theta_H,I,\epsilon_I)$ in a similar manner and consider the rationality of the formal powers $$\wt{F}_{\Theta_H,I,v,\wt{v}}^{\epsilon_I}(T):=\sum_{t\in \CT_{\Theta_H,I}^{--,\epsilon_I}}
\langle \pi(t)v, \wt{v}\rangle T^{|v(t)|}.$$
\end{remark}
From now on, we assume  $\sharp C_H\leq1$. 
\begin{prop}\label{formal power} 
Take $* = \emptyset, +,0,-$ and $\Theta_H\subset\Delta_H$ such that any
 $\Theta_H\subset \Theta_H^\prime\subset \Delta_H$ admits a reduction structure $S(\Theta_H^\prime,*)$. 
 Then for any $v \in \pi^{A_{\emptyset_H}^\circ}$ and $\wt{v} \in \wt{\pi}^{A_{\emptyset_H}^\circ}$, the formal power series
$\wt{F}_{\Theta_H,v,\wt{v}}^*$ is rational. 

Moreover, 
for each $(\Theta,w,s) \in \bigsqcup_{\Theta_H\subset \Theta_H^\prime}S(\Theta_H^\prime,*)$, there exists a one variable polynomial $P_{\Theta,w,s}$ which is
independent of $v$ and $\wt{v}$, and whose first and last coefficients are units  such that   
$$\wt{F}^*_{\Theta_H,v,\wt{v}}(T)=\frac{N(T)}{M(T)\prod_{(\Theta,w,s)}P_{\Theta,w,s}(T^{-|v(s)|})}$$
for some polynomial $N(T) \in R[T]$ and monomial $M(T)\in R[T]$. Here $(\Theta,w,s)$ runs over $\bigsqcup_{\Theta_H\subset \Theta_H^\prime}S(\Theta_H^\prime,*)$.
\end{prop}
\begin{proof}For any $\Theta\subset \Delta_G$, let $P_\Theta$ be the standard parabolic with standard Levi $M_\Theta=Z_{P_\Theta}(A_\Theta)$ and unipotent radical $N_{\Theta}$. We will treat the case $\sharp C_H=1$ here. The case $C_H=\emptyset$ is similar and  easier. By construction, it suffices to show $\wt{F}^*_{\Theta_H,v,\wt{v}}$ is rational for each $*=+,0,-$.  The idea is to make induction on the size of $\Delta_H-\Theta_H$.

We start with $*=0$ (the treatment when $C_H=\emptyset$ is similar to this case). For $\Theta_H=\Delta_H$, $\CT_{\Delta_H}^{--,0}$ is a singleton and $\wt{F}_{\Delta_H,v,\wt{v}}^{0}$ is a constant. For general $\Theta_H \subset \Delta_H$ , we make the induction hypothesis that 
$\wt{F}_{\Theta_H',v,\wt{v}}^0$ is rational for any $\Theta_H \subsetneq \Theta_H' \subset \Delta_H$, $v$ and $\wt{v}$. Now we  prove that $\wt{F}_{\Theta_H,v,\wt{v}}^0$ is rational.

Note that for  $(\Theta,w,s)\in S(\Theta_H,0)$, $\pi(s)$ induces an invertible element in $\End_{R[w^{-1}M_\Theta w(F)]}(J_{w^{-1}N_\Theta w}(\pi))$. Since $\pi$ is a finitely generated smooth admissible $R[G(F)]$-module, $J_{w^{-1}N_\Theta w}(\pi)$ is a finitely generated  smooth admissible $R[w^{-1}M_\Theta(F)w]$-module by Theorem \ref{admjac}.  Thus by \cite[Lemma 4.1.1]{Dis20}, the $R$-module $\End_{R[w^{-1}M_\Theta w(F)]}(J_{w^{-1}N_\Theta w}(\pi))$ is  finite and as explained in \cite[Page 1808]{Mos16},  there exists $$P_{\Theta,w,s}(X)=\sum_{i=0}^Na_iX^i\in R[X],\quad a_0,a_N\in R^\times$$
such that $P_{\Theta,w,s}(\pi(s))v\in \pi(w^{-1}N_\Theta w)$ for all $v\in \pi$.
For $0\leq i\leq N$, assume that 
\begin{equation}\label{decomposition-0}
\CT_{\Theta_H}^{--,0}-T_{s^i}\left(\CT_{\Theta_H}^{--,0}\right)=\bigsqcup_{\Theta_H\subsetneq \Theta_H^\prime}\bigsqcup_{t\in I_{\Theta_H^\prime,i}}T_t\left(\CT_{\Theta_H^\prime}^{--,0}\right)
\end{equation}
  where each $I_{\Theta_H^\prime,i}\subset \CT_{\Theta_H}^{-,0}$ is a finite subset.
  Then one has \begin{align*}
T^{iv(s)}\wt{F}^{0}_{\Theta_H,\pi(s^i)v,\wt{v}}=\wt{F}^{0}_{\Theta_H,v,\wt{v}}-\sum_{\Theta_H^\prime}\sum_{t\in I_{\Theta_H^\prime,i}}T^{v(t)}\wt{F}^{0}_{\Theta_H^\prime,\pi(t)v,\wt{v}}, \ T^{v(t)}:=\prod_{\alpha\in \Delta_H-\Theta_H} T_\alpha^{v_\alpha(t)}
  \end{align*}
Consequently, 
\begin{equation}\label{induction-0}
\wt{F}^{0}_{\Theta_H,P_{\Theta,w,s}(\pi(s))v,\wt{v}}=P_{\Theta,w,s}(T^{-v(s)})\wt{F}^{0}_{\Theta_H, v, \wt{v}}-\sum_{i=1}^N\sum_{\Theta_H^\prime}\sum_{t\in I_{\Theta_H^\prime,i}}a_i T^{v(t)-iv(s)}\wt{F}^{0}_{\Theta_H^\prime,\pi(t)v,\wt{v}}. 
 \end{equation}
 Then by the induction hypothesis, $\wt{F}^{0}_{\Theta_H, v, \wt{v}}$ is rational if $\wt{F}^{0}_{\Theta_H,P_{\Theta,w}(\pi(s))v,\wt{v}}$ is rational. 
 
For any $v\in \pi$, the element
\[v^\prime:=\prod_{(\Theta,w,s)\in S}P_{\Theta,w,s}(\pi(s(\Theta, w)))v\]
lies in $\bigcap_{(\Theta,w,s)\in S}\pi(w^{-1}N_{\Theta}w)$. Then by Corollary \ref{Asym}, $\wt{F}^{0}_{\Theta_H,v^\prime,\wt{v}}$ is a polynomial in the variables $T_{\alpha}$, $\alpha\in \Delta_H-\Theta_H$.

If $S$ has only one element so that $v^\prime = P_{\Theta,w,s}(\pi(s(\Theta, w)))v$, we obtain the rationality of
$\wt{F}^{0}_{\Theta_H,v,\wt{v}}$. If $S$ has more than one element, say, $v^\prime = P_{\Theta,w,s}(\pi(s(\Theta,w)))v_1$ for some
$v_1 \in \pi$, then replacing $v$ by $v_1$ in Equation \ref{induction-0}, we obtain the rationality of
$\wt{F}^{0}_{\Theta_H,v_1,\wt{v}}$ (note that the induction hypothesis includes any $v \in \pi^{A_{\emptyset_H}^\circ}$ and any $\wt{v} \in \wt{\pi}^{A_{\emptyset_H}^\circ}$).
As $S$ is finite, one  finally deduces $\wt{F}^{0}_{\Theta_H,v,\wt{v}}$ is rational.

 By a similar induction argument, the case $*=\pm$ is reduced to the case $*=0$. Precisely, similar to Decomposition
 \ref{decomposition-0} for the case $*=0$, for each $i$, we have
 \begin{equation}\label{decomposition-1}
\CT_{\Theta_H}^{--,*}-T_{s^i}\left(\CT_{\Theta_H}^{--,*}\right)=\left(\bigsqcup_{\Theta_H\subsetneq \Theta_H^\prime}\bigsqcup_{t\in I_{\Theta_H^\prime,i}}T_t\left(\CT_{\Theta_H^\prime}^{--,*}\right)\right) \bigsqcup \left(\bigsqcup_{t_1 \in I_{\Theta_H,i}} T_{t_1}
\left(\CT_{\Theta_H^\prime}^{--,0}\right)\right).
\end{equation}
 where $I_{\Theta_H^\prime,i}$ (resp. $I_{\Theta_H,i}$) is a finite subset of $\CT_{\Theta_H^\prime}^{-,*}$ (resp. 
 $\CT_{\Theta_H}^{-,*}$). Then, similar to Equation \ref{induction-0}, we have
 \begin{equation}\label{induction-1}
 \begin{aligned}
\wt{F}^{*}_{\Theta_H,P_{\Theta,w,s}(\pi(s))v,\wt{v}} =&P_{\Theta,w,s}(T^{-|v(s)|})\wt{F}^{*}_{\Theta_H, v, \wt{v}} 
-\sum_{i=1}^N\sum_{\Theta_H^\prime}\sum_{t\in I_{\Theta_H^\prime,i}}a_i T^{|v(t)|-i|v(s)|}\wt{F}^{*}_{\Theta_H^\prime,\pi(t)v,\wt{v}}\\
&-\sum_{i=1}^{N}\sum_{t_1\in I_{\Theta_H,i}}a_i T^{|v(t)|-i|v(s)|}\wt{F}^{0}_{\Theta_H,\pi(t_1)v,\wt{v}}.
 \end{aligned}
 \end{equation}
 Note that   the rationality of $\wt{F}^{0}_{\Theta_H,\pi(t_1)v,\wt{v}}$ is already established.  The same argument in the $*=0$ case gives
 the rationality for $\wt{F}^{*}_{\Theta_H,\pi(t)v,\wt{v}}$. 

For the "Moreover" part, we just apply Equations \ref{induction-0} and \ref{induction-1} 
by making induction on the size of $\Delta_H - \Theta_H$. We are done.

\end{proof}
Let  $\fn_{\emptyset_H}$ is the Lie algebra of $N_{\emptyset_H}$ and assume $$\det(\Ad_{A_{\emptyset_H}})|_{\fn_{\emptyset_H}}=\sum_{\alpha\in\Delta_H}N_{\alpha}\alpha;\ N_\alpha\in \BN.$$ We will need  to evaluate $\tilde{F}_{\Theta_H,v,\tilde{v}}$ at $T_\alpha=q^{N_\alpha}$, $\alpha\in\Delta_H-\Theta_H$ and $T_\beta=1$, $\beta\in C_H$. To this end, it would be convenient to  change  the formal power series in multiple variables to formal power series in one variable. We record the following corollary  of Proposition \ref{formal power}.   
Note that for any $t\in A_{\Theta_H}$,
$$\delta_{P_{\emptyset_H}}^{-1}(t)=\prod_{\alpha\in\Delta_H-\Theta_H}q^{N_\alpha v_\alpha(t)}.$$
For each $N\in\BN$ and $*=\emptyset,+,0,-$, denote by
$$\CT_{\Theta_H,N}^{--,*}=\left\{t\in\CT_{\Theta_H}^{--,*}\Big|\sum_{\alpha\in(\Delta_H-\Theta_H)\sqcup C_H}|v_\alpha(t)|=N\right\}$$
Note that $\CT_{\Theta_H,N}^{--,*}$ is finite and when  $C_H=\{\beta\}$, $$\CT_{\Theta_H,N}^{--}=\CT_{\Theta_H,N}^{--,+}\bigsqcup \CT_{\Theta_H,N}^{--,0}\bigsqcup \CT_{\Theta_H,N}^{--,-}.$$
\begin{cor}\label{formal power cor}Assumptions and notations as in Proposition \ref{formal power}. Then for  any  $*=\emptyset,+,0,-$ and   $v\in\pi^{A_{\emptyset_H}^\circ}$, $\wt{v}\in\wt{\pi}^{A_{\emptyset_H}^\circ}$,   the formal power series 
$$F^*_{\Theta_H,v,\wt{v}}(S):=\sum_{N\geq 0}\left(\sum_{t\in \CT_{\Theta_H,N}^{--,*}}\langle\pi(t)v,\wt{v}\rangle
\delta_{P_{\emptyset_H}}^{-1}(t)\right)S^N$$
 is rational. Moreover for each $(\Theta_H,*)$, there exists polynomials $Q(S),P(S)\in R[S]$ with the  first and last coefficients of $P(S)$ being units such that 
$$F^*_{\Theta_H,v,\wt{v}}(S)=\frac{Q(S)}{P(S)}.$$
\end{cor}
\begin{proof}The inner summation in the definition of $F^*_{\Theta_H,v,\tilde{v}}$ is finite and thus as formal power series,
$F^*_{\Theta_H,v,\tilde{v}}(S)$ is just the evaluation of $\wt{F}^*_{\Theta_H,v,\tilde{v}}(T)$ at   $T_\alpha=q^{N_\alpha}S$, $\alpha\in \Delta_H-\Theta_H$ and  $T_\beta=S$, $\beta\in C_H$. 

Take the expression $$\wt{F}^*_{\Theta_H,v,\wt{v}}(T)=\frac{N(T)}{M(T)\prod_{(\Theta,w,s)}P_{\Theta,w,s}(T^{-|v(s)|})}$$
as in Proposition \ref{formal power}.
Since $\sum_{\alpha\in  (\Delta_H-\Theta_H)\sqcup C_H}|v_\alpha(s)|>0$, the evaluation of $P_{\Theta,w,s}(T^{-(v(s))|}$ at  $T_\alpha=q^{N_\alpha}S$, $\alpha\in \Delta_H-\Theta_H$ and  $T_\beta=S$, $\beta\in C_H$ is a polynomial in $S$ whose first and last coefficients are units. Consequently, $F^*_{\Theta_H,v,\wt{v}}(S)$ is also rational.
\end{proof}

\begin{defn}\label{controllable}
The $G$-spherical variety $Y$, or spherical pair $(G,H)$,  \dfn{admits a reduction structure} if any $\Theta_H\subset \Delta_H$
admits a reduction structures with respect to $G$.
\end{defn}

\begin{remark}We remark that if  $(G,H)$ and $(G',H')$ are spherical pairs such that
\begin{itemize}
    \item there is a subgroup $Z$ of the center of $G$ such that
$G' = G/Z$, $H' = H/(Z \cap H)$ and $Z\cap H$ is anisotropic;
\item or there is an isogeny $\phi: G \ra G'$ mapping $H$ to $H'$, 
\end{itemize} then the quotient map (resp. $\phi$) identifies $\Delta_H$, $\Rat(A_{\emptyset_H})$, $\Delta_G$  with $\Delta_{H^\prime}$, $\Rat(A_{\emptyset_{H^\prime}})$, $\Delta_{G^\prime}$. Under these identifications,  a reduction structure for $\Theta_H$ with respect to $G$ is sent to a reduction structure for $\Theta_{H^\prime}=\Theta_H$  with respect to $G^\prime$. In particular, admitting a reduction structure is preserved when modifying  spherical pairs by the operations above.
\end{remark}
Now we evaluate $F_{\Theta_H,v,\tilde{v}}$ at $S=1$ when the  pair $(G,H)$ is strongly tempered.
Normalize the Haar measure $dh$ on $H(F)$ so that  $\Vol(K_H,dh)=1$ and take $\omega$, $C_{\Theta_H}$ as in Theorem \ref{Cartan} and Proposition \ref{Vol}. We start with the  following result for  complex tempered representations (which will apply to                        $\pi|_x$, $x\in\Sigma$).
\begin{lem}\label{eva}Assume  the spherical pair $(G,H)$ is strongly tempered and admits a reduction structure. Let $\pi$, $\wt{\pi}$ be  irreducible tempered complex $G(F)$-representations equipped with a non-degenerate bi-$H(F)$-invariant pairing 
$\langle-,-\rangle$.  Take any $v\in\pi^{K_H}$ and $\tilde{v}\in\tilde{\pi}^{K_H}$. Then for $\Theta_H\subset\Delta_H$,
$F_{\Theta_H,v,\tilde{v}}(S)$ is regular at $S=1$ and 
$$F_{\Theta_H,v,\tilde{v}}(1) = \sum_{t\in\CT_{\Theta_H}^{--}}\langle\pi(t)v, \tilde{v}\rangle \delta_{P_{\emptyset_H}}^{-1}(t).$$
Moreover, one has  $$I(v, \tilde{v}):=\sum_{w\in\omega}\sum_{\Theta_H}C_{\Theta_H}F_{\Theta_H,v,\wt{\pi}(w^{-1})\wt{v}}(1)=\int_{H(F)}\langle \pi(h)v, \tilde{v}\rangle dh.$$
\end{lem}
\begin{proof}We deal with the case $\sharp C_H=1$ here. Since the irreducible complex representation $\pi$ (hence $\wt{\pi}$) is  tempered, the period integral $$\int_{H(F)}\langle \pi(h)v, \wt{v}\rangle dh$$
is absolutely convergent. Then by Theorem \ref{Cartan} and Proposition \ref{Vol}, the summation
\begin{align*}
&\sum_{w\in \omega}\sum_{\Theta_H\subset \Delta_H}\sum_{t\in \CT_{\Theta_H}^{--}}C_{\Theta_H}\delta_{P_{\emptyset_H}}^{-1}(t)\int_{K_H\times K_H}\langle \pi(t)\pi(k_2)v,\wt{\pi}(w^{-1}) \wt{\pi}((k_1^{-1})\wt{v}\rangle dk_1 dk_2\\
=& \sum_{w\in\omega}\sum_{\Theta_H\subset \Delta_H}\sum_{t\in\CT_{\Theta_H}^{--}}C_{\Theta_H}\langle\pi(t)v, \wt{\pi}(w^{-1})\wt{v}\rangle \delta_{P_{\emptyset_H}}^{-1}(t)
\end{align*}
is  absolutely convergent and equals to the period integral
  (see \cite[Page 149]{Sil79}). Thus  the summation \begin{equation}\label{formal power complex}
  \sum_{t\in\CT_{\Theta_H}^{--,*}}\langle\pi(t)v, \wt{v}\rangle \delta_{P_{\emptyset_H}}^{-1}(t)=
  \sum_{N\geq0}\sum_{t\in\CT_{\Theta_H,N}^{--,*}}\langle\pi(t)v, \tilde{v}\rangle \delta_{P_{\emptyset_H}}^{-1}(t)
\end{equation}
is also  absolutely convergent for each $\Theta_H\subset\Delta_H$ and $*= +,0,-$.
 By the  absolute convergence of (\ref{formal power complex}), one immediately deduces that the formal power series $F^*_{\Theta_H,v,\wt{v}}$ converges absolutely when $|S|\leq1$. As it is rational, one must have 
 $F^*_{\Theta_H,v,\wt{v}}$ is regular at $S=1$ and 
 $$F^*_{\Theta_H,v,\wt{v}}(1)=\sum_{t\in\CT_{\Theta_H}^{--,*}}\langle\pi(t)v, \tilde{v}\rangle \delta_{P_{\emptyset_H}}^{-1}(t).$$
\end{proof}
Now we turn to Conjecture \ref{conj-2}. We need  the following notation:
\begin{defn}\label{locally twisting} Let $\hat{\Sigma}$ be the pre-image of $\Sigma$ via the natural map $\Spec(R\otimes_E \bar{E})\to\Spec(R)$. 
Let $\pi$ be a finitely generated smooth admissible $R[G(F)]$-module such that for any $x\in\hat{\Sigma}$, $$\CE(\hat{\pi}|_x):=\{\tau:\ \bar{E}\hookrightarrow\BC|\ (\hat{\pi}|_x)_\tau:=\hat{\pi}|_x\otimes_{\bar{E},\tau}\BC\ \mathrm{is}\ \mathrm{tempered}\}\neq\emptyset,\quad \hat{\pi}:=\pi\otimes_E\bar{E}.$$ 
We say the {\em discrete support} of $\hat{\pi}$  is {\em rigid} around $x\in\hat{\Sigma}$ if there exist 
 \begin{itemize}
 \item an open neighborhood $x\in V\subset\Spec(R\otimes_E \bar{E})$,
 \item a smooth character  $\chi: M(F) \to \CO_V^\times$ such that $\chi|_x$ is the trivial character,
    \item a parabolic subgroup   $P=MN\subset G$   with Levi factor $M$,
    \item  an irreducible smooth admissible $M(F)$-representation $\sigma$ over $\bar{E}$ such that $\sigma_\tau$ is a discrete series for some $\tau\in\CE(\hat{\pi}|_x)$,
\end{itemize} 
such that $\hat{\pi}|_y\subset I_P^G\sigma\otimes\chi|_y$ for each $y\in\hat{\Sigma}\cap V$. By abuse of notation, we say the discrete support of $\pi$ is rigid around each $x\in\Sigma$ if the discrete support of $\hat{\pi}$ is rigid around each $x\in\hat{\Sigma}$.
\end{defn}

\begin{thm}\label{contro}Assume the spherical $G$-variety $Y$ is strongly tempered  and admits a reduction structure.
Let $\pi$ be a finitely generated smooth admissible $R[G(F)]$-module as in
Conjecture \ref{conj-2}. Then
\begin{itemize}
    \item the rationality conjecture holds for $\pi$,
    \item the meromorphy conjecture  holds for $\pi|_U$ where $U\subset\Spec(R)$ is a Zariski dense open subset; moreover  the meromorphy conjecture holds for $\pi$ if the discrete support of  $\pi$ is rigid
 around each point $x\in\Sigma$.
\end{itemize}
\end{thm}
\begin{proof}  For any $v\in \pi$ ( resp. $\wt{v}\in \wt{\pi}$), let $$v^\prime=\int_{K_H}\pi(k)vdk\quad \mathrm{resp.}\ \wt{v}^\prime=\int_{K_H}\wt{\pi}(k)\wt{v}dk$$
For any $\Theta_H\subset \Delta_H$, write
$F_{\Theta_H,v^\prime,\wt{v}^\prime}(S)=\frac{Q(S)}{P(S)}$ with polynomials $$P(S)=\sum_{n=0}^N a_n^\prime S^n=\sum_{n=0}^Na_n(S-1)^n,\quad Q(S)=\sum_{n\geq0} b_n(S-1)^n\in R[S]$$
as in Corollary \ref{formal power cor}.  In particular,  $P(S)$ is independent of $v$ and $\wt{v}$ and $a_0^\prime$, $a_N = a_N^\prime\in R^\times$. Moreover by Lemma \ref{eva}, for any $x \in \Sigma$,  if  $r:=\ord_{S=1}P|_x\leq N$ 
is the order of $P|_x$ at $S=1$, then $r\leq\ord_{S=1}Q|_x$. Thus, the evaluation 
$F_{\Theta_H,v^\prime,\wt{v}^\prime}(S)|_{x}(1)$ of $F_{\Theta_H,v^\prime,\tilde{v}^\prime}(S)|_{x}$ at $S=1$  is just 
$\frac{b_{r}(x)}{a_{r}(x)}$.

Firstly we  prove  the rationality conjecture. For any $x\in\Sigma$ and $\tau\in \CE(\pi|_x)$, 
$\frac{\tau(Q|_x)}{\tau(P|_x)}$ represents the rational function 
$F_{\Theta_H, \tau(v^\prime|_x), \tau(\wt{v}^\prime|_x)}$ defined for $\tau(\pi|_x)$. 
Set $$I|_x(v|_x,\wt{v}|_x):=\sum_{w\in\omega}\sum_{\Theta_H\subset\Delta_H}C_{\Theta_H}F_{\Theta_H,v^\prime,\wt{\pi}(w^{-1})\wt{v}^\prime}|_{x}(1)$$
Then by Proposition \ref{SRP} and  Lemma \ref{eva}, the rationality conjecture holds for $\pi|_x$ and $I|_x$ is the desired  bi-$H(F)$-invariant bilinear form on $\pi|_x\times\wt{\pi}|_x$.

Now we consider the meromorphy conjecture. Let $\fp_1,\cdots,\fp_{k}$ be the minimal primes of $R$ . For each $1\leq i \leq k$, there exists $0\leq r(i)\leq N$ such that 
$$a_0,\cdots, a_{r(i)-1}\in\fp_i,\quad a_{r(i)}\notin\fp_i.$$
 Set $U=\Spec(R)-Z$ where $Z=\bigcup_{\Theta_H,i}Z(\Theta_H,i)$ with $Z(\Theta_H,i)=V(\fp_i)\cap V(a_{r(i)})\subset\Spec(R)$, $1\leq i\leq k$. Let $\Frac(R)=\prod_i R_{\fp_i}$ be the total quotient ring
of $R$.
Set 
$$I(v,\wt{v}):=\sum_{w\in\omega}\sum_{\Theta_H}C_{\Theta_H}F_{\Theta_H,v^\prime,\wt{\pi}(w^{-1})\wt{v}^\prime}(1)\in \Frac(R),\ F_{\Theta_H,v^\prime,\wt{\pi}(w^{-1})\wt{v}^\prime}(1):= \left(\frac{b_{r(i)}}{a_{r(i)}}\right)_{i}\in \Frac(R).$$
Then for any $x\in\Sigma\cap U$, one has $$I(v,\wt{v})|_x=I|_x(v|_x,\wt{v}|_x),\quad \forall\ v\in\pi,\wt{v}\in\wt{\pi}.$$
Since $\Sigma\cap U$ is Zariski dense in $\Spec(R)$, $I(-,-)$ is forced to be bi-$H(F)$-invariant and bilinear by the interpolation property. 
Thus, the meromorphy conjecture holds for $\pi|_U$.

 For the meromorphy conjecture of $\pi$, it suffices to show $I(-,-)$ 
actually interpolates $I|_x(-,-)$ for each $x\in \Sigma$. For this, one can base change from $R$ to $R\otimes_E \bar{E}$ and consider the corresponding result for $\hat{\pi}$.  Clearly for $x\in\hat{\Sigma}$, $I(-,-)$ (defined for $\hat{\pi}$)
 interpolates $I|_x(-,-)$ if for all $\Theta_H\subset \Delta_H$ and $*=+,0,-$ (resp. $\emptyset$) if $\sharp C_H=1$ (resp. $C_H=\emptyset$), $F^*_{\Theta_H,v,\wt{v}}(1)$ (defined for $\hat{\pi}$) is regular at $x$. By Proposition \ref{formal power} and (the proof) of Corollary  \ref{formal power cor},  $F_{\Theta_H,v,\wt{v}}^*(1)$ is regular at $x$ if one can choose $P_{\Theta,w,s}$ such that 
$$ P_{\Theta,w,s}|_x(\delta_{P_{\emptyset_H}}(s))\neq0,\ \forall\ (\Theta,w,s)\in \bigsqcup_{\Theta_H\subset \Theta_H^\prime}S(\Theta_H^\prime,*).$$

Assume the discrete support of $\hat{\pi}$ is rigid around each $x \in \hat{\Sigma}$ and take $V$, $\sigma$ and $\chi$ as in Definition \ref{locally twisting}. Then 
by the theory of Jacquet modules (see \cite[Section 6.3]{Cas95}),  for any $(\Theta,w,s)$
\begin{itemize}
    \item the central characters  of  irreducible sub-quotients of $J_{w^{-1}N_\Theta w}(I_P^G\sigma\otimes\chi|_y)$ are 
  interpolated by $\CO_V^\times$-valued characters, say $\{\omega_i\}_{i\in I}$;
  \item the multiplicity of irreducible sub-quotients of $J_{w^{-1}N_\Theta w}(I_P^G\sigma\otimes\chi|_y)$  with  the same central character are uniformly bounded,  say by $\{N_i\}_i$,
\end{itemize}
when $y$ varies in $V$.
Then locally around $x$,  $\prod_{i\in I}(\hat{\pi}(s)-\omega_i(s))^{N_i}$  kills $J_{w^{-1}N_\Theta w}(\hat{\pi})$ and we can take   $P_{\Theta,w,s}(T)=\prod_{i\in I}(T-\omega_i(s))^{N_i}$. Since  $I_P^G\sigma_\tau$ is a direct sum of tempered representations for some $\tau\in\CE(\hat{\pi}|_x)$, one has $P_{\Theta,w,s}|_x(\delta_{P_{\emptyset_H}}(s))\neq0$ by   Lemma \ref{eigenvalue} below.
\end{proof}
\begin{lem}\label{eigenvalue}Let $(G,H)$ be a strongly tempered spherical pair  and  $\pi$ be an  irreducible tempered complex $G(F)$-representation. Let $*$ be one of $+,0,-$ if $\sharp C_H=1$ or $\emptyset$ if $C_H=\emptyset$. Let $\Theta_H \subset \Delta_H$ and 
 consider a triple $(\Theta,w,s)$ where  $\Theta\subset \Delta_G$, $w\in W$ and 
$s\in w^{-1}A_{\Theta}^-w\cap A_{\emptyset_H}^-$ with image in $\CT_{\emptyset_H}^-$ belongs to $\CT_{\Theta_H}^--\{e\}$. Assume that
for some $\epsilon>0$ small enough, $$\left\{a\in A_{\Theta_H}^{--,*} \Big| \ \mathrm{if}\ waw^{-1}\in A_{\emptyset_G}^-,\ \exists \alpha
\in \Delta_G-\Theta, |\alpha(waw^{-1})|\geq \epsilon\right\}\neq A_{\Theta_H}^{--,*}.$$
Then any eigenvalue for the action of $\pi(s)$ on the Jacquet module $J_{w^{-1}N_\Theta w}(\pi)$ has absolute value strictly smaller than $\delta_{P_{\emptyset_H}}(s)$.  
 \end{lem}
 \begin{proof}Let $\{\chi_i\}_{i\in I}$ be the set of central characters of irreducible sub-quotients of the finite length smooth admissible $w^{-1}M_\Theta(F) w$-representation $J_{w^{-1}N_\Theta w}(\pi)$. For each $\chi_i$, let $N_i$ be the multiplicity of the irreducible sub-quotients with central character $\chi_i$. Consider the generalized eigenspace
 $$V_i:=\left\{v\in J_{w^{-1}N_\Theta w}(\pi) \Big| (\pi(a)-\chi_i(a))^{N_i}v=0,\ \forall\ a\in  w^{-1}A_{\Theta}w\right\}.$$
 Then as $w^{-1}A_{\Theta}w$-representations, 
 $$J_{w^{-1}N_\Theta w}(\pi)=\bigoplus_i V_i.$$
By assumption, there exists $a\in A_{\Theta_H}^{--,*}$ such that $waw^{-1}\in A_{\emptyset_G}^-$ and for any $\alpha\in\Delta_G-\Theta$, $|\alpha(waw^{-1})|<\epsilon$.
Let $(-,-)$ be the contraction on $\pi\times\pi^\vee$. 
By Theorem \ref{cano},
$$(\pi(as^n)v,v^\vee)=(\pi(as^n)v,v^\vee)_{w^{-1}N_\Theta w},\quad \forall\ n\geq0,\ v\in \pi,\ v^\vee\in \pi^\vee.$$
Assume that $v$ (as an element in  $J_{w^{-1}N_\Theta w}(\pi)$) decomposes as  $$v=\sum_i v_i,\quad v_i\in V_i.$$
On one hand, similar to Lemma \ref{eva}, the summation 
$$\sum_{n\geq0}(\pi(as^n)v,v^\vee)\delta_{P_{\emptyset_H}}^{-1}(as^n)$$
converges absolutely as  $\pi$ is tempered and $(G,H)$ is strongly tempered. 
On the other hand,
$$(\pi(as^n)v_i,v^\vee)_{w^{-1}N_\Theta w}=\chi_i(s)^{n-N_i}\sum_{j=0}^{N_i-1} {{n}\choose {j}}
\chi_i^{N_i-j}(s)\left((\pi(s)-\chi_i(s))^j\pi(a)v_i, v^\vee\right),\quad n\geq N_i.$$
Thus for each $i$, one must have 
$$|\chi_i(s)\delta_{P_{\emptyset_H}}^{-1}(s)|<1.$$
We are done.
\end{proof}
We expect the discrete support of  $\pi$ to  be rigid around each $x\in\Sigma$   in practice.  The following proposition is due to  Prof. G. Moss.
\begin{prop}\label{quasi-split} Assume $E=\bar{E}$ and let $\pi$ be a finitely generated  smooth admissible torsion-free $R[G(F)]$-module such that
\begin{itemize}
    \item $\pi|_x$ is irreducible tempered  for each $x\in\Sigma$,
    \item the natural morphism $R\to\End_{R[G(F)]}(\pi)$ is an isomorphism.
\end{itemize} Then the discrete support of  $\pi$ is rigid around  $x\in\Sigma$ if $\pi|_x=I_P^G\sigma$ where $P=MN\subset G$ is a parabolic subgroup with Levi factor $M$ and $\sigma$ is an irreducible  regular supercuspidal $M(F)$-representation over $E$.
\end{prop}
\begin{proof}Base change along any $\tau\in\CE(\pi|_x)$, we may assume $E=\BC$. Since $\pi|_x$ is tempered,   $\pi|_x\hookrightarrow I_{P^\prime}^G\tau$ for some parabolic subgroup $P^\prime=M^\prime N^\prime \subset G$ and discrete series $M^\prime(F)$-representation $\tau$. The $M^\prime(F)$-representation $\tau$ also has cuspidal support $\sigma$. Thus by \cite[Theorem 6.5.1]{Cas95} for $\tau$, one has $\sigma$ is unitary (hence a discrete series). Let $Z(G)$ and $Z(M)$ be the Bernstein center for smooth $G(F)$ and $M(F)$-representations. Since $R\cong \End_{R[G(F)]}(\pi)$, there exists a morphism $Z(G)\to R$ such that for any $y\in\Spec(R)$, the composition 
$$Z(G)\to R \to k(y)$$
factors through the action of $Z(G)$ on $\pi|_y$. Let $\fs$ (resp. $\ft$ ) be the $G$-inertial (resp. $M$-inertial) class of $[M,\sigma]$   and $Z(G)_{\fs}$ (resp. $Z(M)_\ft$) be the corresponding component.  Then up to replacing $\Spec(R)$ by its connected component containing $x$, we may assume $Z(G)\to R$ factors through $Z(G)_\fs\to R$.

 The Harish-Chandra morphism $Z(G)\to Z(M)$ (see \cite[Theorem 4.1]{DHK22}) induces a finite morphism $Z(G)_\fs\to Z(M)_\ft$, which is locally \'etale at the point $\sigma$ by the regularity assumption.  Locally around $\sigma$, we can take a section and by abuse of notation, we denote this section  by  $Z(M)_\ft\to Z(G)_\fs$. Then the composition
$$Z(M)_\ft\to Z(G)_\fs\to R\to k(x)$$
is given by the action of $Z(M)_\ft$ on $\sigma$. Let $M_0(F)\subset M(F)$ be the intersection of the kernels of all unramified characters of $M(F)$. Then $M(F)/M_0(F)$ is a free abelian group of finite rank. 
Let $$\chi:\ M(F)\to \BC[M(F)/M_0(F)];\quad m\mapsto [m]$$ be the universal unramified character of $M(F)$. Since $\sigma$ is supercuspidal, $Z(M)_\ft$ is just the ring of functions on the orbit of $\sigma$ under unramified twisting. Hence there exists an open neighborhood   $x\in V\subset \Spec(R)$  such that for any  $y\in V$, the composition
$$Z(M)_\ft\to Z(G)_\fs\to R\to k(y)$$
is simply given by the action of $Z(M)_{\ft}$ on $\sigma\otimes\chi|_{y}$. Then by the definition of the Harish-Chandra morphism,  the cuspidal support of  $\pi|_y$ is $\sigma\otimes\chi|_{y}$ for $y\in V$.
Then for $y\in \Sigma\cap  V$,  $\pi|_y\hookrightarrow I_P^G\sigma\otimes\chi|_{y}$ and the desired result follows.
\end{proof}

When $G$ is a product of general linear groups, the above result can be improved thanks to the theory of extended Bernstein varieties.
\begin{prop}\label{GL(n)}Assume $G$ is a product of general linear groups and  $E=\bar{E}$. Let $\pi$ be a finitely generated smooth admissible torsion-free $R[G(F)]$-module such that $\pi|_x$ is irreducible and tempered at any $x\in\Sigma$. Then the discrete support of $\pi$ is rigid around each $x\in\Sigma$.
\end{prop}
\begin{proof}We will briefly indicate how to deduce this result from the theory of co-Whittaker modules and extended Bernstein varieties summarized in  \cite{Dis20}. We will use the notations in \textit{loc.cit}. Note that $\pi|_x$ is generic at each $x\in\Sigma$.
Assume $G=\prod_{i} G_i$ with $G_i=\GL_{n_i}$ for some $n_i\in\BN$. 
Let $\pi_i$ be the top derivative of $\pi$ with respect to $\prod_{j\neq i}G_j$. Then by \cite[Lemma 3.1.5]{EH14},  $\pi_i$ is a finitely generated  smooth admissible torsion-free $\CO_X[G_i]$-module for each $i$. By \cite[Lemma 4.2.1,4.2.3,4.2.5]{Dis20}, up to shrinking $\Spec(R)$ to an open subset containing $\Sigma$,  $\pi$ and $\pi_i$ are all co-Whittaker. Moreover by \cite[Lemma 4.2.6]{Dis20}, one has $\pi=\otimes_{i=1}^n\pi_i$ up to furthur shrinking the open subset. Thus we are  reduced  to the case $G=\GL_n$, which we assume below.

 By \cite[Lemma 4.2.5]{Dis20}, $\End_{R[G(F)]}(\pi)=R$. For each generic point $\eta\in\Spec(R)$, $ \pi|_\eta$  is absolutely irreducible by Lemma \ref{irre} and thus 
 determines an inertial type $[\fs]$ and a multi-partition $[t]$  by \cite[Section 2.2]{Dis20}.  Then by \cite[Section 3.3]{Dis20}, up to replacing $\Spec(R)$ by its connected component containing $x$, there is a
 map $\alpha$ from $\Spec(R)$ to  the component $\fX_{[\fs],[t]}^\prime$ of the extended Bernstein variety.
Let  $\fM$ be the torsion-free co-Whittaker module, whose construction is explained in the proof of \cite[Lemma 5.2.2]{Dis20}, over $\fX_{[\fs],[t]}'$. Now by \cite[Lemma 4.2.6]{Dis20},   $\alpha^*\fM\cong \pi$ on an open  neighborhood of $x\in\Sigma$. 

By the construction of $\fM$ and the classification results of $\GL_n(F)$-representations, we are done. 
\end{proof}

The most important source of  spherical varieties admitting reduction structures is strongly tempered spherical varieties without type
$N$-roots.
\begin{thm}\label{crucial} For $G$ split, all strongly  tempered spherical $G$-varieties  without type $N$-roots admit reduction structures.
\end{thm}
\begin{proof}
 According to \cite[Section 1.1]{WZ21}, 
 a  strongly tempered spherical pair  without type $N$-roots for $G$-split is  essentially  one of  the following
\begin{align*}
    &\GL_n\subset \GL_{n+1}\times\GL_n;\quad  \SO_n\subset \SO_{n+1}\times \SO_n;\\
    & \Delta\BG_m\bs \GL_2\times \GL_2\subset \Delta \BG_m\bs \GL_4\times\GL_2;\quad \Sp_4\times\Sp_2\subset \Sp_6\times\Sp_4;
   \end{align*}
or the pair $ \BG_m \subset \PGL_2$, which is already  dealt with in Example \ref{Wal}. Before analyzing the remaining four pairs   case by case, we explain the rough idea to produce a reduction structure  for $\Theta_H\subset \Delta_H$ and $*=+,0,-$ (resp. $*=\emptyset$) if $\sharp C_H= 1$ (resp. $C_H=\emptyset$):
\begin{itemize}
\item firstly, choose elements $w\in W$ such that $A_{\Theta_H}^{--,*}$ can be decomposed into a disjoint union $\bigsqcup_w A_{\Theta_H,w}^{--,*}$ where  $A_{\Theta_H,w}^{--,*}\neq\emptyset$  and $wA_{\Theta_H,w}^{--,*} w^{-1}\subset A_{\emptyset_G}^-$. For example, one can take the single element $w=e$ if and only if $A_{\Theta_H}^{--,*} \subset A_{\emptyset_G}^-$;
\item secondly, choose several $\Theta\subset\Delta_G$ for each $w$ such that for $\epsilon>0$ small enough, the set  
$$\bigcap_\Theta \left\{a\in A_{\Theta_H,w}^{--,*} \Big| \exists \alpha\in \Delta_G-\Theta, |\alpha(waw^{-1})|\geq \epsilon\right\}$$
is finite modulo $A_{\Theta_H}^\circ$. In this step, we always choose $\Theta$ as large as possible. 
\item lastly, choose the element $s$ for each pair $(\Theta,w)$.
\end{itemize}

  \s{(1) The $(\GL_{n+1}\times \GL_n, \GL_n)$-case.} Consider the pair
  $$H:=\GL_n \subset G:=\GL_{n+1}\times \GL_n, \quad g\mapsto (\diag\{g,1\},g).$$
  Let $B_n$ be the Borel subgroup of $\GL_n$ of upper-triangler matrices with $A_n\subset B_n$ the subgroup of  diagonal matrices. 
 Let  $P_{\emptyset_G}=B_{n+1}\times B_n$, $A_{\emptyset_G}=A_{n+1}\times A_n$ and
$P_{\emptyset_H}=B_n$, $A_{\emptyset_H}= A_n\subset A_{\emptyset_G}$. Note that $\Delta_{\GL_n}=\{\alpha_i| 1\leq i\leq n-1\}$ where for $1\leq i \leq n-1$
$$\alpha_i:\ A_n\to\BG_m;\quad (t_1,\cdots,t_n)\mapsto t_i/t_{i+1}.$$
Identify $\Delta_G$ with $\Delta_{\GL_{n+1}}\sqcup \Delta_{\GL_n}$.  Let $\beta_i$ (resp. $\alpha_i$) be  the $i$-th simple root of $\GL_{n+1}$ (resp. $\GL_n$).  
 Let $w_i\in W_G$, $1\leq i\leq n+1$ be the element
\begin{align*}
    & w_i(\diag\{t_1,\cdots, t_{i-1},t_i,\cdots, t_{n+1}\}, \diag\{a_1,\cdots, a_n\})w_i^{-1}\\
    =& (\diag\{t_1,\cdots,t_{i-1}, t_{n+1},t_i,\cdots,t_n\}, \diag\{a_1,\cdots, a_n\})
\end{align*}
Note that $w_{n+1}=e$. Let $T(i,\varpi):=\diag\{\varpi,\cdots,\varpi,1\cdots,1\} \in A_n(F)$ where the  first $i$-terms are $\varpi$.

 Let $C_H$ consists of the character 
$$ A_n\to\BG_m;\quad (t_1,\cdots,t_n)\mapsto t_n.$$
With respect to this choice, $A_{\emptyset_H}^{-,+}\sqcup A_{\emptyset_H}^{-,0}\subset A_{\emptyset_G}^{-}$.
 Now we  give a reduction structure $S(\Theta_H,*)$ for each
$\Theta_H\subset\Delta_H$ and $*=+,0,-$. 

 For $*=0$ and $\Theta_H\subset\Delta_H$, the set 
   $$S(\Theta_H,0)=\{\left(\Delta_G-\{\beta_i,\alpha_i\},w_{n+1}, T(i,\varpi)\right)\}_{\alpha_i\in \Delta_H-\Theta_H}$$
   is a reduction structure since
   \begin{itemize}
       \item for any $0< \epsilon\leq 1$,
    \begin{align*} &\bigcap_{(\Theta,w,s)\in S(\Theta_H,0)}\{a\in A_{\Theta_H}^{--,0}\bigg| \mathrm{if}\ waw^{-1}\in A_{\emptyset_G}^-,\ \exists \alpha\in \Delta_G-\Theta, |\alpha(waw^{-1})|\geq \epsilon\}\\
    =&\{a\in A_{\Theta_H}^{--,0}\bigg| |a_i/a_{i+1}|\geq \epsilon, \ \alpha_i\in \Delta_H-\Theta_H\}
    \end{align*}
    has finite image in $\CT_{\Theta_H}^{--,0}$;
    \item for each $(\Theta,w,s)\in S(\Theta_H,0)$, $s\in w^{-1}A_\Theta^-w\cap A_{\Theta_H}^{-,0}$ 
 and for any $n\in\BN$,  $\CT_{\Theta_H}^{--,0}-T_{s^n}(\CT_{\Theta_{H}}^{--,0})$ is a finite disjoint union of  translations of $\CT_{\Theta_H\cup\{\alpha_i\}}^{--,0}$.
   \end{itemize}
 For  $*=+$ and $\Theta_H\subset\Delta_H$, the set 
   $$S(\Theta_H,+)=\{(\Delta_G-\{\beta_i,\alpha_i\},w_{n+1}, T(i,\varpi)\}_{\alpha_i\in \Delta_H-\Theta_H}\bigsqcup\{(\Delta_G-\{\beta_n\},w_{n+1}, T(n,\varpi)\}$$
   is a reduction structure since
 \begin{itemize}
       \item for any $0<\epsilon\leq1$, \begin{align*} &\bigcap_{(\Theta,w,s)\in S(\Theta_H,+)}\{a\in A_{\Theta_H}^{--,+}\bigg| \mathrm{if}\ waw^{-1}\in A_{\emptyset_G}^-,\ \exists \alpha\in \Delta_G-\Theta, |\alpha(waw^{-1})|\geq \epsilon\}\\
    =&\{a\in A_{\Theta_H}^{--,+}\bigg| |a_i/a_{i+1}|\geq \epsilon, \ \alpha_i\in \Delta_H-\Theta_H, |a_n|\geq\epsilon\}
    \end{align*}
    has finite image in $\CT_{\Theta_H}^{--,+}$;
    \item for each $(\Theta,w,s)\in S(\Theta_H,+)$, $s\in w^{-1}A_\Theta^-w\cap A_{\Theta_H}^{-,+}$ 
 and for any $n\in\BN$,  $\CT_{\Theta_H}^{--,+}-T_{s^n}(\CT_{\Theta_{H}}^{--,+})$ is a finite disjoint union of  translations of $\CT_{\Theta_H\sqcup\{\alpha_i\}}^{--,+}$ (resp. $\CT_{\Theta_H}^{--,0}$ ) if $\Theta\neq\Delta_G-\{\beta_n\} $ (resp. $\Theta=\Delta_G-\{\beta_n\} $)
   \end{itemize}
For the case $*=-$, let $S=S(\Theta_H,-)$ be the disjoint union of the  set
\begin{align*}
    &\{\left(\Delta_G-\{\beta_i,\alpha_i\},w_{j}, T(i,\varpi)\right)\}_{\stackrel{\alpha_i\in \Delta_H-\Theta_H} {i+1\leq j}}\bigsqcup\{\left(\Delta_G-\{\beta_{i+1},\alpha_{i}\},w_{j}, \varpi^{-1}T(i,\varpi)\right)\}_{\stackrel{\alpha_{i}\in \Delta_H-\Theta_H}{ j\leq i+1}}  
\end{align*}
where $2\leq j\leq n$ and $\alpha_{j-1}\notin\Theta_H$, and the set
\[\{\left(\Delta_G-\{\beta_{i+1},\alpha_{i}\},w_{1}, \varpi^{-1}T(i,\varpi)\right)\}_{\alpha_{i}\in \Delta_H-\Theta_H} \bigsqcup \{\left(\Delta_G-\{\beta_{1}\},w_{1}, \varpi^{-1}T(0,\varpi)\right)\}. \]
Here the terms with  $w=w_j$ are designed to deal with the region $\{a\in A_{\Theta_H}^{--,-}\big| |a_{j-1}|\leq 1< |a_j|\}$. The set $S(\Theta_H,-)$ is a reduction structure since  
 \begin{itemize}
       \item for any $0<\epsilon\leq1$, \begin{align*} &\bigcap_{(\Theta,w,s)\in S}\{a\in A_{\Theta_H}^{--,-}\bigg| \mathrm{if}\ waw^{-1}\in A_{\emptyset_G}^-,\ \exists \alpha\in \Delta_G-\Theta, |\alpha(waw^{-1})|\geq \epsilon\}\\
    =&\bigsqcup_{\stackrel{1\leq j\leq n}{\alpha_{j-1}\notin\Theta_H}}\bigcap_{(\Theta,w,s)\in S}\{a\in A_{\Theta_H}^{--,-}\bigg||a_{j-1}|\leq1< |a_j|\ \&\  \mathrm{if}\ waw^{-1}\in A_{\emptyset_G}^-,\ \exists \alpha\in \Delta_G-\Theta, |\alpha(waw^{-1})|\geq \epsilon\}\\
    =&\bigsqcup_{\stackrel{1\leq j\leq n}{\alpha_{j-1}\notin\Theta_H}} \bigcap_{\alpha_i\in \Delta_H-\Theta_H}\{a\in A_{\Theta_H}^{--,-}\bigg| |a_i/a_{i+1}|\geq \epsilon,\ i+1\neq j;\quad \epsilon\leq|a_{j-1}|\leq1<|a_j|\leq \epsilon^{-1},\ i+1=j\}
    \end{align*}
    has finite image in $\CT_{\Theta_H}^{--,-}$. Here by convention, $\alpha_0\notin\Theta_H$ and $\epsilon\leq|a_0|\leq1$ hold.
    \item for each $(\Theta,w,s)\in S$, $s\in w^{-1}A_\Theta^-w\cap A_{\Theta_H}^{-,+}$ 
 and for any $n\in\BN$,  $\CT_{\Theta_H}^{--,-}-T_{s^n}(\CT_{\Theta_{H}}^{--,-})$ is a finite disjoint union of  translations of $\CT_{\Theta_H\sqcup\{\alpha_i\}}^{--,-}$ (resp. $\CT_{\Theta_H}^{--,0}$ ) if $\Theta\neq\Delta_G-\{\beta_1\} $ (resp. $\Theta=\Delta_G-\{\beta_1\} $)
   \end{itemize}

 \s {(2) The $(\SO_{n+1}\times \SO_n, \SO_n)$-case.}
 Let $V_{n+1}$ be a non-degenerate  quadratic $F$-space of dimension $n+1$ and $v \in V_{n+1}$ be an anisotropic vector. Let $V_n=\langle v\rangle^\perp\subset V_{n+1}$. Let  $\SO_{n+1}=\SO(V_{n+1})$ and identify  $\SO_n= \SO(V_n)$ with the stabilizer of $v$. Let $G =   \SO_{n+1} \times \SO_{n}$ and embed $H=\SO_n$  into $G$ diagonally.

When $n=1$, $H$ is anisotropic. When $n=2$, we can proceed  as in Example \ref{Wal} as $\SO(3)\cong \PGL(2)$ and $\SO(2)$ is a torus. For $n\geq3$, fix an orthogonal  decomposition
$$V_n=X_n\oplus W_n\oplus Y_n;\quad V_{n+1}= X_{n+1}\oplus W_{n+1}\oplus Y_{n+1}$$
such that $X_n\subset X_{n+1}$ and $Y_n\subset Y_{n+1}$ are totally isotropic subspaces of the same dimension  and $W_n, W_{n+1}$ are anisotropic.
By fixing  bases of $X_n$ and $X_{n+1}$ compatibly, we obtain minimal parabolic subgroups $P_n\subset \SO_n$ and $P_{n+1} \subset \SO_{n+1}$ with maximal split tori $A_n$ and $A_{n+1}$ respectively such that $P_n\subset P_{n+1}$ and $A_n\subset A_{n+1}$. Take $A_{\emptyset_G}=A_{n+1}\times A_n\subset P_{\emptyset_G}=P_{n+1}\times P_n$,  $A_{\emptyset_H}=A_n\subset P_{\emptyset_H}= P_n$.

Let $r_i:=\dim_F X_i$. Then for $i=n,n+1$,
\begin{itemize}
    \item the Levi factor of $P_i$ is isomorphic to  $(F^\times)^{r_i}\times \SO(W_i)$ and  $A_i$ is isomorphic to $(F^\times)^{r_i}$;
    \item the set of simple roots $\Delta_{\SO_i}$ consists of characters
    $$\alpha_j^i:\ A_i\to \BG_m;\quad (a_1,\cdots, a_{r_i})\mapsto a_j/a_{j+1} \quad \forall\ 1\leq j\leq r_i-1$$
    together with
    \begin{itemize}
        \item the character $\alpha_{r_i}^i(a)= a_{r_i}$ if $\dim_F W_i=1$;
        \item the character $\alpha_{r_i}^i(a)= a_{r_i-1}a_{r_i}$ if $\dim_F W_i=0$.
    \end{itemize}
\end{itemize}
 Identify $\Delta_G $ with $\Delta_{\SO_{n+1}}\sqcup \Delta_{\SO_n}$ and write $\alpha_j^{n+1}$ (resp. $\alpha_j^{n}$) as $\beta_j$ (resp. $\alpha_j$). For any $1\leq j\leq r_n$, let $T(j,\varpi)=(\varpi,\cdots,\varpi,1\cdots,1)\in A_{\Theta_H}^-$ with the  first $j$-terms  being $\varpi$.

When $r_{n+1}=r_n+1$, or equivalently $\dim_F W_n=1$,  one has $A_{\emptyset_H}^-\subset A_{\emptyset_G}^-$. In this case, $\Theta_H\subset\Delta_H$ admits  the reduction structure 
$$\{(\Delta_G-\{\beta_j,\alpha_j\},e, T(j,\varpi))_{\alpha_j\in \Delta_H-\Theta_H, j<r_n}, (\Delta_G-\{\beta_{r_n}, \beta_{r_n+1},\alpha_{r_n}\},e, T(r_n,\varpi))_{\alpha_{r_n}\in \Delta_H-\Theta_H}\}.$$

 When $r_{n+1}=r_n$, or equivalently $\dim_F W_n=0$, let $w\in W_G$ be the element such that
 $$w^{-1}\left( (a_1,\cdots,a_{r_n-1},a_{r_n}), (b_1,\cdots,b_{r_n})\right)w=\left( (a_1,\cdots,a_{r_n-1},a_{r_n}^{-1}), (b_1,\cdots,b_{r_n})\right)$$
 Then for $\Theta_H\subset\Delta_H$,
 \begin{itemize}
     \item if $\alpha_{r_n-1}\in\Theta_H$, then $A_{\Theta_H}^-\subset A_{\emptyset_G}^-$ and $\Theta_H$ has the reduction structure
      $$ \{(\Delta_G-\{\beta_j,\alpha_j\},e, T(j,\varpi))_{\alpha_j\in \Delta_H-\Theta_H}\},$$
      \item if $\alpha_{r_n-1}\notin \Theta_H$ and $\alpha_{r_n}\in\Theta_H$, then $w^{-1}A_{\Theta_H}^-w\subset A_{\emptyset_G}^-$ and $\Theta_H$ has the reduction structure
      $$ \{(\Delta_G-\{\beta_j,\alpha_j\},w, T(j,\varpi))_{\alpha_j\in \Delta_H-\Theta_H, j\leq r_n-2},\ (\Delta_G-\{\beta_{r_n}, \alpha_{r_n-1}\},w, T(r_n-1,1))\}$$
      where $T(r_n-1,1)=(\varpi,\cdots,\varpi,\varpi^{-1})\in A_{\Theta_H}^-$ with the first $(r_n-1)$-terms being $\varpi$,
      \item if $\alpha_{r_n-1}\notin \Theta_H$ and $\alpha_{r_n}\notin\Theta_H$, then $S(\Theta_H)$ admits the reduction structure
      \begin{align*}
         &\bigsqcup\{(\Delta_G-\{\beta_j,\alpha_j\},e, T(j,\varpi)),(\Delta_G-\{\beta_j,\alpha_j\},w, T(j,\varpi))\}_{\alpha_j\in \Delta_H-\Theta_H, j\leq r_n-2} \\
         \bigsqcup&\{(\Delta_G-\{\beta_{r_n-1},\alpha_{r_n-1},\alpha_{r_n}\},e, T(r_n-1,\varpi)), (\Delta_G-\{\beta_{r_n-1},\alpha_{r_n-1},\alpha_{r_n}\},w, T(r_n-1,\varpi))\}\\
          \bigsqcup&\{(\Delta_G-\{\beta_{r_n},\alpha_{r_n}\},e, T(r_n,\varpi)), (\Delta_G-\{\beta_{r_n},\alpha_{r_n-1}\},w, T(r_n-1,1)\}
      \end{align*} 
 \end{itemize}

\s {(3) The $\Delta\BG_m\bs \GL_2\times \GL_2\subset \Delta \BG_m\bs \GL_4\times\GL_2$-case.} For the pair
$$H:=\Delta\BG_m\bs \GL_2\times \GL_2 \hookrightarrow G:=\Delta\BG_m\bs \GL_4\times \GL_2,\quad (g_1,g_2)\mapsto (\diag\{g_1,g_2\},g_2),$$
take  $A_{\emptyset_G}=\Delta\BG_m\bs A_4\times A_2\subset P_{\emptyset_G}=\Delta\BG_m\bs B_4\times B_2$  and $A_{\emptyset_H}=\Delta\BG_m\bs A_2\times A_2\subset P_{\emptyset_H}=\Delta\BG_m\bs B_2\times B_2$. Identify
 $\Delta_G$ with $\Delta_{\GL_4}\sqcup \Delta_{\GL_2}$ and  $\Delta_H$ with $\Delta_{\GL_2}\sqcup \Delta_{\GL_2}$.
 For $i=1,2,3$, let $\beta_i\in\Delta_{\GL_4}\subset \Delta_G$ be the $i$-th simple root.
Let $\alpha$ (resp. $\alpha_i$, $i=1,2$) be the unique simple root in $\Delta_{\GL_2}\subset \Delta_G$ (the $i$-th copy of $\Delta_{\GL_2}$ in $\Delta_H$).
For any permutation $abcd$ of $1234$, let $w_{abcd}\in W_G$ be the element such  that
$$w_{abcd}((t_1,t_2,t_3,t_4),(a_1,a_2))w_{abcd}^{-1}=((t_{a},t_{b},t_{c},t_{d}),(a_1,a_2)).$$
Denote by $a_i \in A_2(F)$, $i=0,1,2$ with
\[a_0 = \matrixx{1}{0}{0}{1}, \quad a_1 = \matrixx{\varpi}{0}{0}{1}, \quad a_2 = \matrixx{\varpi}{0}{0}{\varpi}\]
and $s_{ij} = (a_i,a_j) \in A_{\emptyset_H}(F)$, $0 \leq i,j \leq 2$. Let $C_H$ consist of the character
$$\BG_m\bs A_2\times A_2\to\BG_m,\quad ((t_1,t_2),(t_3,t_4))\mapsto t_2/t_3.$$
With respect to this choice, $A_{\emptyset_H}^{-,0}\sqcup A_{\emptyset_H}^{-,+}\subset A_{\emptyset_G}^{-}$. Then
for any $\Theta_H \subset \Delta_H$ and $*=+,0,-$, one can choose the following reduction structure:  the terms with  $w=w_{abcd}$ are designed to deal with the region $\{((t_1,t_2),(t_3,t_4))\in A_{\Theta_H}^{--,-}\big|\ |t_a|\leq |t_b|\leq |t_c|\leq |t_d|\}$

\begin{table}[h!]
  \caption{Reduction Structure for $\Delta\BG_m\bs \GL_2\times \GL_2\subset \Delta \BG_m\bs \GL_4\times\GL_2$, $*=0,+$}
  \begin{center}
    \begin{tabular}{l|l|l} 
      \toprule
      $\Theta_H$ & Members in $S(\Theta_H,0)$ & Members in $S(\Theta_H,+)$  \\
      \midrule
      \multirow{1}{*}{$\Delta_H$} &  $S(\Delta_H,0)=\emptyset$ & $\big(\Delta_G-\{\beta_2\}, w_{1234},s_{20}\big)$ \\
      \midrule
      \multirow{2}{*}{$\{\alpha_1\}$} & $\big(\Delta_G-\{\beta_3,\alpha\}, w_{1234},s_{21}\big)$ &  $\big(\Delta_G-\{\beta_2\}, w_{1234},s_{20} \big)$\\
   & & $ \big(\Delta_G-\{\beta_3,\alpha\}, w_{1234},s_{21}\big)$ \\
       \midrule
      \multirow{2}{*}{$\{\alpha_2\}$} & $\big(\Delta_G-\{\beta_1\}, w_{1234},s_{10}\big)$ & $\big(\Delta_G-\{\beta_2\}, w_{1234},
      s_{20} \big)$\\
      & & $\big(\Delta_G-\{\beta_1\}, w_{1234},s_{10}\big)$\\
      \midrule
      \multirow{3}{*}{$\emptyset_H$} & $
      \big(\Delta_G-\{\beta_1\}, w_{1234},s_{10}\big)$ &  $\big(\Delta_G-\{\beta_2\}, w_{1234},s_{20}\big)$\\
 &     $\big(\Delta_G-\{\beta_3,\alpha\}, w_{1234},s_{21}\big)$ &
        $\big(\Delta_G-\{\beta_1\}, w_{1234},s_{10}\big)$\\
   &  & $\big(\Delta_G-\{\beta_3,\alpha\}, w_{1234},s_{21}\big)$\\
      \bottomrule
    \end{tabular}
  \end{center}
\end{table}

\begin{table}[h!]
  \caption{Reduction Structure for $\Delta\BG_m\bs \GL_2\times \GL_2\subset \Delta \BG_m\bs \GL_4\times\GL_2$, $*=-$}
  \begin{center}
    \begin{tabular}{l|l} 
      \toprule
      $\Theta_H$ & Members in $S(\Theta_H)$ (each line with the same Weyl element)  \\
      \midrule
      \multirow{1}{*}{$\Delta_H$} & 
       $\big(\Delta_G-\{\beta_2\}, w_{3412},s_{02}\big)$ \\
      \midrule
      \multirow{2}{*}{$\{\alpha_1\}$} 
      & $\big(\Delta_G-\{\beta_1,\alpha\}, w_{3124}, s_{01}\big),
      \big(\Delta_G-\{\beta_3,\alpha\}, w_{3124}, s_{21}\big)$ \\
      & $\big(\Delta_G-\{\beta_1,\alpha\}, w_{3412}, s_{01}\big),
      \ \big(\Delta_G-\{\beta_2\}, w_{3412}, s_{02}\big)$ \\
      \midrule
      \multirow{2}{*}{$\{\alpha_2\}$}  &     $\big(\Delta_G-\{\beta_1\}, w_{1342},s_{10}\big),\ \big(\Delta_G-\{\beta_3\}, w_{1342},s_{12}\big)$ \\ 
      & $\big(\Delta_G-\{\beta_3\}, w_{3412},s_{12}\big),\ \big(\Delta_G-\{\beta_2\}, w_{3412},s_{02}\big)$\\
       \midrule
      \multirow{5}{*}{$\emptyset_H$} 
      & $\big(\Delta_G-\{\beta_1\}, w_{1324},s_{10}\big),\
      \big(\Delta_G-\{\beta_2,\alpha\}, w_{1324},s_{11}\big),\
      \big(\Delta_G-\{\beta_3,\alpha\}, w_{1324},s_{21}\big)$ \\
       & $\big(\Delta_G-\{\beta_1\}, w_{1342},s_{10}\big),\
      \big(\Delta_G-\{\beta_3\}, w_{1342} ,s_{12}\big),\
      \big(\Delta_G-\{\beta_2,\alpha\}, w_{1342},s_{11}\big)$ \\
      & $\big(\Delta_G-\{\beta_1,\alpha\}, w_{3124},s_{01}\big),\
      \big(\Delta_G-\{\beta_2,\alpha\}, w_{3124},s_{11}\big),\
      \big(\Delta_G-\{\beta_3,\alpha\}, w_{3124},s_{21} \big)$ \\
      & $\big(\Delta_G-\{\beta_1,\alpha\}, w_{3142},s_{01}\big),\
      \big(\Delta_G-\{\beta_2,\alpha\}, w_{3142},s_{11}\big),\
      \big(\Delta_G-\{\beta_3\}, w_{3124},s_{12}\big)$ \\
      & $\big(\Delta_G-\{\beta_1,\alpha\}, w_{3412} ,s_{01}\big),\
      \big(\Delta_G-\{\beta_3\}, w_{3412},s_{12}\big),\
      \big(\Delta_G-\{\beta_2\}, w_{3412},s_{02}\big)$ \\
      \bottomrule
    \end{tabular}
  \end{center}
\end{table}

\s {(4) The $\Sp_4\times\Sp_2\subset \Sp_6\times\Sp_4$-case.} Let $w_n\in \GL_n(F)$ be the anti-diagonal matrix with entries $1$. For $n=2m$, define the split symplectic group
$$\Sp_n=\{g\in\GL_{2m}| {}^tg J_{2m} g=J_{2m}\}\quad J_{2m}=\begin{pmatrix} 0 & -w_m\\ w_m & 0\end{pmatrix}.$$
 Let $A_n\subset B_n\subset \Sp_n$  be the maximal split torus of diagonal matrices and the Borel subgroup  of upper triangular matrices respectively. Then under the isomorphism
$$\BG_m^m\cong A_n;\quad (t_1,\cdots, t_m)\mapsto (t_1,\cdots, t_m,t_m^{-1},\cdots,t_1^{-1}),$$
 $\Delta_{\Sp_n}=\{\alpha_i|i=1,\cdots, m\}$ where
  $\alpha_i((t_1,\cdots,t_m))=t_i/t_{i+1}$ (resp. $t_i^2$) if $i< m$( resp. $i=m$).

For the pair $$H:=\Sp_4\times \Sp_2\hookrightarrow G:=\Sp_6\times\Sp_4,\quad
 \left(h_1, h_2=\begin{pmatrix} a & b \\ c & d\end{pmatrix}\right)\mapsto\left( \begin{pmatrix} a & 0 & b\\ 0 & h_1 & 0\\ c & 0 & d\end{pmatrix}, h_1\right),$$
take $A_{\emptyset_H}=A_4\times A_2\subset P_{\emptyset_H}=B_4\times B_2$ and $A_{\emptyset_G}=A_6\times A_4\subset P_{\emptyset_G}=B_6\times B_4$. Identify $\Delta_H$ with $\Delta_{\Sp_4}\sqcup \Delta_{\Sp_2}$ and $\Delta_G$ with $\Delta_{\Sp_6}\sqcup \Delta_{\Sp_4}$. Denote by $\gamma_i$ (resp. $\beta_i$, resp. $\alpha_i$) the $i$-th simple root of $A_6$ (resp. $A_4$, resp. $A_2$).
Let $w_1=e$ and $w_2,w_3\in W_G$ such that
$$w_2((t_1,t_2,t_3), (a_1,a_2))w_2^{-1}=((t_2,t_1,t_3),(a_1,a_2)),\quad w_3((t_1,t_2,t_3), (a_1,a_2))w_3^{-1}=((t_2,t_3,t_1),(a_1,a_2)).$$

Denote by $a_i \in A_4(F)$, $i=0,1,2$ with
\[a_0 = (1,1), \quad a_1 = (\varpi,1), \quad a_2 = (\varpi,\varpi)\]
and $s_{ij} = (a_i,\varpi^j) \in A_{\emptyset_H}(F)$, $i=0,1,2$ and $j=0,1$.
Then for any $\Theta_H \subset \Delta_H$, one can choose the following reduction structure: the terms with $w=w_i$, $i=1,2,3$ are designed to deal with the region 
$$\{\left((t_1,t_2),t_3\right)\in A_{\Theta_H}^{--}\ \big|\ |t_3|\leq|t_1|\leq|t_2|\leq 1,\ \mathrm{resp.}\ |t_1|\leq|t_3|\leq|t_2|\leq 1,\ \mathrm{resp.}\  |t_1|\leq|t_2|\leq|t_3|\leq 1\}$$
\begin{table}[h!]
  \caption{Reduction Structure for $\Sp_4\times\Sp_2\subset \Sp_6\times\Sp_4$}
  \begin{center}
    \begin{tabular}{l|l} 
      \toprule
      $\Theta_H$ & Members in $S(\Theta_H)$ (each line with the same Weyl element)  \\
      \midrule
      \multirow{1}{*}{$\Theta_H \supset \{\alpha_1\}$} &
      $(\Delta_G-\{\gamma_i,\beta_i\},w_3,s_{i0})$ with $i$ satisfying $\beta_i\in \Delta_H-\Theta_H$ \\
      \midrule
      \multirow{1}{*}{$\{\beta_1,\beta_2\}$} & $\big(\Delta_G-\{\gamma_1\},e, s_{01} \big)$ \\
      \midrule
      \multirow{2}{*}{$\{\beta_1\}$} & $\big(\Delta_G-\{\gamma_1\},e,s_{01}\big),\
      \big(\Delta_G-\{\gamma_3,\beta_2\}, e,s_{21}\big )$ \\
      & $\big(\Delta_G-\{\gamma_2,\beta_2\},w_3, s_{20}\big),\
      \big(\Delta_G-\{\gamma_3,\beta_2\}, w_3,s_{21}\big)$ \\
      \midrule
      \multirow{2}{*}{$\{\beta_2\}$} & $\big(\Delta_G-\{\gamma_1\},e,s_{01}\big),\
      \big(\Delta_G-\{\gamma_2,\beta_1\}, e, s_{11}\big)$\\
      &$\big(\Delta_G-\{\gamma_1,\beta_1\},w_2,s_{10}\big),\
      \big(\Delta_G-\{\gamma_2,\beta_1\}, w_2, s_{11}\big)$ \\
      \midrule
      \multirow{3}{*}{$\emptyset_H$} & $\big(\Delta_G-\{\gamma_1\},e,s_{01}\big),\
      \big(\Delta_G-\{\gamma_2,\beta_1\}, e,s_{11}\big),\
      \big(\Delta_G-\{\gamma_3,\beta_2\}, e,s_{21}\big)$ \\
      &$\big(\Delta_G-\{\gamma_1,\beta_1\},w_2,s_{10} \big),\
      \big(\Delta_G-\{\gamma_2,\beta_1\}, w_2,s_{11}\big),\
      \big(\Delta_G-\{\gamma_3,\beta_2\}, w_2, s_{21}\big)$ \\
      &$\big(\Delta_G-\{\gamma_1,\beta_1\},w_3, s_{10}\big),\
      \big(\Delta_G-\{\gamma_2,\beta_2\},w_3,s_{20}\big),\
      \big(\Delta_G-\{\gamma_3,\beta_2\}, w_3, s_{21}\big)$  \\
    \bottomrule
    \end{tabular}
  \end{center}
\end{table}

\end{proof}


 Proposition \ref{crucial} deals with {\em split} spherical varieties.  For non-split ones, we  record the following:

\begin{prop}\label{ggp-unitary}
Let  $K/F$ be a quadratic field extension. Let $V_{n+1}$ be a non-degenerate Hermitian space with
respect to $K/F$ of dimension $n+1$.  Let  $\RU_{n+1}=\RU(V_{n+1})$ be the unitary group of $V_{n+1}$. Let $v \in V_{n+1}$ be an anisotropic vector and
$V_n=\langle v\rangle^\perp\subset V_{n+1}$. Identify
$\RU_n= \RU(V_n)$ with the stabilizer of $v$. Let $G =   \RU_{n+1} \times \RU_{n}$ and embed $H=\RU_n$
into $G$ diagonally. Then all the statements in Theorem \ref{main} holds for the strongly tempered spherical pair $(G,H)$.
\end{prop}
\begin{proof}One can show that $(G,H)$ admits a reduction structure
similarly as the case $(\SO_{n+1} \times \SO_n, \SO_n)$.
\end{proof}

\s{\bf Declarations}
On behalf of all authors, the corresponding author states that there is no conflict of interest.

\s{\bf Statements}
Data sharing not applicable to this article as no datasets were generated or analysed during the current study.

\end{document}